\documentclass[journal,twocolumn]{IEEEtran}
\IEEEoverridecommandlockouts
\usepackage{cite}
\usepackage{amsmath,amssymb,amsfonts,amsthm}
\usepackage{graphicx}

\usepackage{subcaption}
\usepackage{mathtools} % coloneqq
\usepackage[linesnumbered, ruled]{algorithm2e} % Algorithm writing
\usepackage{textcomp}
\usepackage{tabu} % Spacing out tables
\usepackage{xcolor}
\usepackage[font={small}]{caption}

\newtheorem{theorem}{Theorem}
\newtheorem{prop}{Proposition}
\newtheorem{lem}{Lemma}

\usepackage{comment}

\def\BibTeX{{\rm B\kern-.05em{\sc i\kern-.025em b}\kern-.08em
	T\kern-.1667em\lower.7ex\hbox{E}\kern-.125emX}}

\newcommand{\jp}[1]{\textcolor{blue}{Josh: #1}}

\renewcommand{\Re}{\mathbb{R}}
\renewcommand{\d}{\mathrm{d}}

\graphicspath{{./Figures/}}

\SetKwRepeat{Do}{do}{while}%

\begin{document}

\title{Covariance Steering with Optimal Risk Allocation\\}
\author{%
	Joshua Pilipovsky\textsuperscript{*} \thanks{\textsuperscript{*} J. Pilipovsky is a Graduate student at the School of Aerospace Engineering, Georgia Institute of Technology, Atlanta, GA 30332-0150, USA. Email: jpilipovsky3@gatech.edu} 
	~~~~
	~~~~
	Panagiotis Tsiotras\textsuperscript{$\dagger$} \thanks{\textsuperscript{$\dagger$} P. Tsiotras is the David \& Lewis Chair and Professor at the School of Aerospace Engineering and the Institute for Robotics \& Intelligent Machines, Georgia Institute of Technology, Atlanta, GA 30332-0150, USA. Email: tsiotras@gatech.edu}
}%
\maketitle

\begin{abstract}
	This paper extends the optimal covariance steering problem for linear stochastic systems subject to chance constraints so as  
	to account for an optimal allocation of the risk.
	Previous works have assumed a \textit{uniform} risk allocation  in order to cast the optimal control problem as a semi-definite program (SDP), which can be solved efficiently using standard SDP solvers. 
	An  Iterative Risk Allocation (IRA) formalism is used to solve the optimal risk allocation problem for covariance steering using a two-stage approach. 
	The upper-stage of IRA optimizes the risk, which is  a convex problem, while the lower-stage optimizes the controller with the new constraints.
	The process is applied 
	iteratively until the optimal risk allocation that achieves the lowest total cost is obtained.
	The proposed framework results in solutions that tend to maximize the terminal covariance, while still satisfying the chance constraints, thus leading to less conservative solutions than previous methodologies.
	In this work, we consider both polyhedral and cone state chance constraints.
	%	For the case of the latter, we introduce two novel convex relaxation methods to approximate quadratic chance constraints as second-order cone constraints. 
	%	Additionally, we incorporate \textit{hard} input constraints, extending the work of~\cite{IH} for polyhedral chance constraints.
	Finally, we demonstrate the approach to a spacecraft rendezvous problem and compare the results with other competing approaches.
\end{abstract} 

\section{Introduction}

In this paper we address the problem of finite-horizon stochastic optimal control of a discrete linear time-varying (LTV) system with white-noise Gaussian diffusion. 
The control task is to steer the state from an initial Gaussian distribution to a final Gaussian distribution with known statistics. 
In addition to the boundary conditions, we consider chance constraints that restrict the probability of violating state constraints to be less than a certain threshold.
In general, 
hard state constraints are difficult to impose in stochastic systems because the noise can be unbounded, so chance constraints are typically used to deal with this problem by imposing a small, but finite, probability of violating the constraints. 
In the literature, one encounters two kinds of chance constraints; \textit{individual} chance constraints and \textit{joint} chance constraints~\cite{b3}. 
Individual chance constraints limit the probability of violating each constraint, while joint chance constraints limit the probability of violating \textit{any} constraint over the whole time horizon. 
In this paper, we consider the case of joint chance constraints, because they are a more natural choice for most applications of interest.

Control of stochastic systems can be best formulated as a problem of controlling the distribution of trajectories over time. 
Moreover, Gaussian distributions are completely characterized by their first and second moments, so the control problem can be thought of as one of 
steering the mean and the covariance to the desired terminal values. 
The problem of covariance control has a history dating back to the '80s, with the works of Hotz and Skelton \cite{HS1,HS2}. 
Much of the early work focused on the \textit{infinite} horizon problem, where the state covariance \textit{asymptotically} approaches its terminal value.
Only recently has the finite-horizon problem drawn attention, with much of the early work focusing on the so-called covariance steering (CS) problem,
namely, the problem of steering an initial distribution to a final distribution at a specific final time step subject to linear time-varying (LTV) dynamics. 
The problem can be thought of as a linear-quadratic Gaussian (LQG) problem with a condition on the terminal covariance~\cite{EB2}.
Moreover, it has been shown that the finite-horizon covariance steering controller can be constructed as a state-feedback controller, and the problem 
can be formulated either as a convex program~\cite{EB2,EB3}, or as the solution of a pair of Lyapunov differential equations coupled through their boundary conditions~\cite{Halder,Chen}.
Alternatively, for certain special cases one can solve the CS problem directly by solving an LQ stochastic problem with a particular choice of cost weights~\cite{Max1}.
Other approaches~\cite{EB1,PP} use an affine-disturbance feedback controller having two components, one that steers the mean state and the other that steers the covariance. 

Steering the covariance is related to the theory of steering marginal distributions, which has a long history, stemming from the problem of Schr\"{o}dinger bridges and optimal mass transport~\cite{Chen,Chen2,Chen3,OT}. 
Most recent work on covariance steering  has focused on incorporating physical constraints on the system, such as state chance constraints~\cite{Max2}, obstacles in path-planning environments~\cite{PP}, input hard constraints~\cite{IH}, incomplete state information~\cite{incompleteState}, and extensions in the context of stochastic model predictive control~\cite{CSMPC} and nonlinear systems~\cite{b2,NL2,NL3}.

In this work, we first review the formalism for the Covariance Steering Chance Constraint (CSCC) problem, to account for \textit{optimal} risk allocation.
By risk allocation we mean allocating the probability of violating each individual chance constraint at each time step. 
For example, if there are $M$ chance constraints and $N$ time steps, there would be $NM$ total allocations for the whole problem. 
Previous works~\cite{Max1,Max2,PP,IH,CSMPC} have assumed a constant risk allocation, so that the resulting problem can be turned into a semi-definite program (SDP). 
Here, however, we adopt an iterative two-stage algorithm that optimizes the risk distribution over all time steps, and subsequently optimizes the controller by solving a SDP.
Other works have tried to optimize the risk using techniques such as ellipsoidal relaxation\cite{ER} and particle control\cite{PC}. 
However, ellipsoidal relaxation techniques are overly conservative and lead to highly suboptimal solutions. 
Particle control methods are computationally too demanding, since the number of decision variables grows with the number of particles. 
The two-stage risk allocation scheme proposed in this work is computed iteratively until the cost is within a given tolerance of the minimum, from which we get the \textit{optimal} risk allocation for the problem, as well as the optimal controller.

The contributions of this work are as follows:
First, we propose the first work that solves the covariance steering problem while optimally allocating the risk of violating chance constraints.
Second, we extend previous similar results in the literature that only deal with polyhedral constraints to the case of cone constraints.
Third, we provide a result of independent interest for convexifying the chance constraints for the case of conical regions.
%Fourth, using the results from \cite{IH}, we incorporate \textit{hard} input constraints into the problem formulation for both polyhedral and cone chance constraints.
Finally, and in order to illustrate the proposed risk allocation algorithm, we use as an example a rendezvous problem between two spacecraft, in which the approaching spacecraft's control inputs are constrained to be within some bounds, while its position should remain within a predetermined LOS region during the whole maneuver.
Both polyhedral and cone LOS constraints are investigated and compared. 

This work extends the preliminary results of~\cite{JP1} along several directions.
First, while in \cite{JP1} only polyhedral chance constraints  were handled, in this work we also treat the important -- and much more difficult --  
case of cone constraints. 
Indeed, in many applications the constraints are given in the form of a conical region (e.g., line-of-sight (LOS) constraints).
Approximating such cone constraints with many intersecting hyperplanes would make the problem computationally expensive and quite challenging in terms of achieving high accuracy approximations.
Second, we also present a way to approximate  such cone chance constraints as special cases of general \textit{quadratic constraints} in terms of two-sided polyhedral constraints. 
Third, we apply this formulation to the case of LOS cone chance constraints, and compare with a polyhedral approximation.
Fourth, we present a geometric relaxation of the cone chance constraints, which is more natural in real-life applications.
Finally, we compare our results with stochastic MPC-type approaches.

The paper is structured as follows: 
In Section~II we define the general stochastic optimal control problem for steering a distribution from an initial Gaussian to a terminal Gaussian with joint state chance constraints. 
In Section~III we review the two-stage risk allocation formalism, and formulate the SDP for the optimal controller as well as the proposed iterative risk allocation algorithm.
In Section~IV we present two different convex relaxations of quadratic chance constraints, one in terms of a two-sided linear constraint relaxation, and the other based on
a geometrical construction.
Finally, in Section~V we implement the theory to the spacecraft rendezvous and docking problem with both polyhedral and cone chance constraints, and compare with stochastic MPC methods.

%---------------------------------------------------%
\section{Problem Statement}

We consider the following discrete-time stochastic time-varying system subject to noise
\begin{equation}~\label{eq:1}
	x_{k+1} = A_kx_k + B_ku_k + D_kw_k,
\end{equation}
where $x\in\mathbb{R}^n,u\in\mathbb{R}^m$, with time steps $k = 0,\ldots,N-1$, where $N$ representing the finite horizon. 
The uncertainty $w\in\mathbb{R}^r$ is a zero-mean white Gaussian noise with unit covariance, i.e., $\mathbb{E}[w_k] = 0$ and $\mathbb{E}[w_{k_1}w_{k_2}^\intercal] = I_{r}\delta_{k_1,k_2}$. Additionally, we assume that $\mathbb{E}[x_{k_1}w_{k_2}^\intercal] = 0$, for $0\leq k_1 \leq k_2 \leq N$. The initial state $x_0$ is a random vector drawn from the normal distribution
\begin{equation}~\label{eq:2}
	x_0 \sim \mathcal{N}(\mu_0,\Sigma_0),
\end{equation}
where $\mu_0\in\mathbb{R}^n$ is the initial state mean and $\Sigma_0\in\mathbb{R}^{n\times n} > 0$ is the initial state covariance. The objective is to steer the trajectories of (\ref{eq:1}) from the initial distribution (\ref{eq:2}) to the terminal distribution
\begin{equation}~\label{eq:3}
	x_f \sim \mathcal{N}(\mu_f,\Sigma_f),
\end{equation}
where $\mu_f\in\mathbb{R}^n$ and $\Sigma_f > 0$ are the state mean and covariance at time $N$, respectively. The cost function to be minimized is 
\begin{equation}~\label{eq:4}
	J(u_0,\ldots,u_{N-1}) \coloneqq \mathbb{E}\bigg[\sum_{k=0}^{N-1}x_k^\intercal Q_k x_k + u_k^\intercal R_k u_k\bigg],
\end{equation}
where $Q_k\geq 0$ and $R_k > 0$ for all $k = 0,\ldots,N-1$. Additionally, and over the \textit{whole} horizon, we impose the following joint chance constraint that limits the probability of state violation to be less than a pre-specified threshold, i.e.,
\begin{equation}~\label{eq:5}
	\mathbb{P}\bigg(\bigwedge_{k=0}^{N} x_k\notin\mathcal{X}\bigg) \leq \Delta,
\end{equation}
where $\mathbb{P}(\cdot)$ denotes the probability of an event, $\mathcal{X}\subset\mathbb{R}^n$ is the state constraint set, and $\Delta\in(0,0.5]$.

\noindent
\textit{Remark 1}: We assume that the system (\ref{eq:1}) is controllable, that is, for any $x_0,x_f\in\mathbb{R}^n$, and no noise ($w_k\equiv 0,~ k = 0,\ldots,N-1$), there exists a sequence of control inputs $\{u_k\}_{k=0}^{N-1}$ that steer the system from $x_0$ to $x_f$.

First, we provide an alternative description of the system (\ref{eq:1}) in order to solve the problem at hand. Using \cite{Max1,Max2,PP,IH,CSMPC}, we can reformulate (\ref{eq:1}) as
\begin{equation}~\label{eq:6}
	X = \mathcal{A}x_0 + \mathcal{B}U + \mathcal{D}W,
\end{equation} 
where $X\coloneqq [x_0^\intercal,...x_N^\intercal]^\intercal\in\mathbb{R}^{(N+1)n}, U\coloneqq[u_0^\intercal,...u_{N-1}^\intercal]^\intercal\in\mathbb{R}^{Nm}$, and $W\coloneqq [w_0^\intercal,...,w_{N-1}^\intercal]^\intercal\in\mathbb{R}^{Nr}$ are the state, input, and disturbance sequences, respectively. The matrices $\mathcal{A},\mathcal{B}$, and $\mathcal{D}$ are defined accordingly \cite{Max1}. Using this notation, we can write the cost function compactly as 
\begin{equation}~\label{eq:7}
	J(U) = \mathbb{E}[X^\intercal\bar{Q}X + U^\intercal\bar{R}U], 
\end{equation}
where $\bar{Q}$ and $\bar{R}$ are defined accordingly. 
Note that since $Q_k \geq 0$ and $R_k > 0$ for all $k = 0,\ldots,N-1$, it follows that $\bar{Q} \geq 0$ and $\bar{R} > 0$. 
The initial and terminal conditions (\ref{eq:2}) and (\ref{eq:3}) can be written as
\begin{equation}~\label{eq:8}
	\mu_0 = E_0\mathbb{E}[X], \qquad \Sigma_0 = E_0\Sigma_XE_0,
\end{equation}
and
\begin{equation}~\label{eq:9}
	\mu_f = E_N\mathbb{E}[X], \qquad \Sigma_f = E_N\Sigma_XE_n,
\end{equation}
where $\Sigma_X \coloneqq \mathbb{E}[XX^\intercal] - \mathbb{E}[X]\mathbb{E}[X]^\intercal$, and 
$E_k \coloneqq [0_{n,kn},I_{n},0_{n,(N-k)n}]$ picks out the $k$th component of a vector.
Consequently, the state chance constraints (\ref{eq:5}) can be written as 
\begin{equation}~\label{eq:10}
	\mathbb{P}\bigg(\bigwedge_{k=1}^{N}E_kX\notin\mathcal{X}\bigg) \leq \Delta.
\end{equation}
In summary, we wish to solve the following stochastic optimal control problem.
\\ \\
\noindent
\textit{Problem 1}~\label{problem:1}: Given the system (\ref{eq:6}), find the optimal control sequence $U^* \coloneqq U^*_{N-1}$ that minimizes the objective function (\ref{eq:7}), subject to the initial state (\ref{eq:8}), terminal state (\ref{eq:9}), and the state chance constraints (\ref{eq:10}). 
%---------------------------------------------------%

\section{Chance Constrained Covariance Steering with Risk Allocation}

\subsection{Lower-Stage Covariance Steering}

Borrowing from the work in \cite{PP}, we adopt the control policy 
\begin{equation}~\label{eq:11}
	u_k = v_k + K_ky_k,
\end{equation}
where $v_k\in\mathbb{R}^m,K_k\in\mathbb{R}^{m\times n}$, and $y_k\in\mathbb{R}^n$ is given by
\begin{subequations}~\label{eq:12}
	\begin{align}
		y_{k+1} &= A_ky_k + D_kw_k,\\
		y_0 &= x_0 - \mu_0.
	\end{align}
\end{subequations}
\textit{Remark 2}: The proposed control scheme (\ref{eq:11})-(\ref{eq:12}) leads to a convex programming formulation of Problem 1 as follows.

Using (\ref{eq:11})-(\ref{eq:12}), we can write the control sequence as 
\begin{equation}~\label{eq:13}
	U = V + KY,
\end{equation}
where $Y := \mathcal{A}y_0 + \mathcal{D}W\in\mathbb{R}^{(N+1)n}, V\coloneqq [v_0^\intercal,\ldots,v_{N-1}^\intercal]^\intercal\in\mathbb{R}^{Nm}$ and $K\in\mathbb{R}^{Nm\times (N+1)n}$
a matrix containing the gains $K_k$. 
It follows that the dynamics can be decoupled into a mean and error state as follows
\begin{align}
	\bar{X} &= \mathbb{E}[X] = \mathcal{A}\mu_0 + \mathcal{B}V, \label{eq:14}\\
	\tilde{X} &= X - \bar{X} = (I+\mathcal{B}K)Y.\label{eq:15}
\end{align}
Additionally, by noting that the cost function (\ref{eq:7}) can be decoupled into a mean cost $J_{\mu} = \bar{X}^\intercal\bar{Q}\bar{X} + \bar{U}^\intercal\bar{R}\bar{U}$ and a covariance cost $J_{\Sigma} = \textrm{tr}(Q\Sigma_{X}) + \textrm{tr}(Q\Sigma_{U})$, the cost function takes the form \cite{PP}
\begin{align}~\label{eq:16}
	J(V,K) = &(\mathcal{A}\mu_0 + \mathcal{B}V)^\intercal \bar{Q}(\mathcal{A}\mu_0 + \mathcal{B} V) + V^\intercal \bar{R} V\nonumber\\
	+ \ \text{tr}\Big[\big( &(I+\mathcal{B}K)^\intercal\bar{Q}(I+\mathcal{B}K) + K^\intercal\bar{R} K\big)\Sigma_Y\Big],
\end{align}
where $\Sigma_Y\coloneqq \mathcal{A}\Sigma_0\mathcal{A}^\intercal + \mathcal{D}\mathcal{D}^\intercal$. The terminal constraints can be reformulated as 
\begin{subequations}
	\begin{align}
		\mu_f &= E_N(\mathcal{A}\mu_0+\mathcal{B}V), \label{eq:17a}\\
		\Sigma_f &= E_N(I+\mathcal{B}K)\Sigma_Y(I+\mathcal{B}K)^\intercal E_N^\intercal. \label{eq:17b}
	\end{align}
\end{subequations}
Qualitatively speaking, $V$ steers the mean of the system to $\mu_f$, while $K$ steers the covariance to $\Sigma_f$. In order to make the problem convex, we relax the terminal covariance constraint (\ref{eq:17b}) to $\Sigma_f \geq E_N(I+\mathcal{B}K)\Sigma_Y(I+\mathcal{B}K)^\intercal E_N^\intercal$, which can be written as the linear matrix inequality (LMI)
\begin{equation}~\label{eq:18}
	\begin{bmatrix}
		\Sigma_f & E_N(I+\mathcal{B}K)\Sigma_Y^{1/2}\\
		\Sigma_Y^{1/2}(I+\mathcal{B}K)^\intercal E_N^\intercal & I
	\end{bmatrix}
	\geq 0.
\end{equation}

\subsection{Polyhedral Chance Constraints}

When dealing with the risk allocation problem, it is customary to assume that the state constraint set $\mathcal{X}$ is a \textit{convex} polytope $\mathcal{X}^p$, so that
% In Section IV, we provide a novel approach to deal with \textit{quadratic} chance constraints which arise from convex cones and are more natural. 
% Under the assumption of polytopic constraints, however, it is noted \cite{PP} that non-convex constraint sets can also be handled by decomposing $\mathcal{X}^p$ as 
% \begin{equation*}
% \mathcal{X}^p = \mathcal{X}_{\Omega}\setminus\Big(\bigcup_{j=1}^{N_{\text{obs}}}\mathcal{X}_j\Big),
% \end{equation*}
% where $\mathcal{X}_\Omega,\mathcal{X}_1,...,\mathcal{X}_{N_{\text{obs}}}$ are convex polytopes and $N_{\text{obs}}$ is the number of obstacles. Moreover, we assume that the polytope $\mathcal{X}^p$ in our case can be written as a finite intersection of $M$ linear inequalities, that is,
\begin{equation}~\label{eq:19}
	\mathcal{X}^p \coloneqq \bigcap_{j=1}^{M} \{x:\alpha_j^\intercal x \leq \beta_j\},
\end{equation}
where $\alpha_j\in\mathbb{R}^{n}$ and $\beta_j\in\mathbb{R}$. Under this assumption, the probability of violating the state constraints (\ref{eq:10}) can be written as 
\begin{equation}~\label{eq:20}
	\mathbb{P}\bigg(\bigwedge_{k=1}^{N}\bigwedge_{j=1}^{M} \alpha_j^\intercal E_k X > \beta_j\bigg) \leq \Delta.
\end{equation}
Equation (\ref{eq:20}) represents the objective that the joint probability of violating any of the $M$ state constraints over the horizon $N$ is less than or equal to $\Delta$. 
Using Boole's Inequality \cite{Boole,BM}, one can decompose a \textit{joint} chance constraint into \textit{individual} chance constraints as follows
\begin{equation}~\label{eq:22}
	\mathbb{P}\big(\alpha_j^\intercal E_k X \leq \beta_j \big) \geq 1 - \delta_k^j, ~~~ k = 1,\ldots,N,~ j = 1,\ldots,M.\textsl{}
\end{equation}
with
\begin{equation}~\label{eq:21}
	\sum_{k = 1}^{N}\sum_{j=1}^{M}\delta_k^j \leq \Delta,
\end{equation} 
where each $\delta_k^j$ represents the probability of violating the $j$th constraint at time step $k$. 
Notice that the probability in (\ref{eq:22}) is of a random variable with mean $\alpha_j^\intercal E_k\bar{X}$ and covariance $\alpha_j^\intercal E_k\Sigma_XE_k^\intercal \alpha_j$. 
Thus, (\ref{eq:22}) can be equivalently written as 
\begin{equation}~\label{eq:23}
	\Phi\Bigg(\frac{\beta_j-\alpha_j^\intercal E_k\bar{X}}{\sqrt{\alpha_j^\intercal E_k\Sigma_XE_k^\intercal \alpha_j}}\Bigg) \geq 1 - \delta_k^j,
\end{equation}
where $\Phi(\cdot)$ denotes the cumulative distribution function of the standard normal distribution. Simplifying (\ref{eq:23}) and noting that $\Sigma_X = (I+\mathcal{B}K)\Sigma_Y(I+\mathcal{B}K)^\intercal$ yields
\begin{align}~\label{eq:24}
	\alpha_j^\intercal &E_k(\mathcal{A}\mu_0+\mathcal{B}V)\nonumber\\
	&+\|\Sigma_Y^{1/2}(I+\mathcal{B}K)^\intercal E_k^\intercal\alpha_j\|\Phi^{-1}(1-\delta_k^j)\leq \beta_j.
\end{align}
\textit{Remark 3}: Since $\Sigma_0 > 0$, it follows that $\Sigma_Y > 0$, and $\Sigma_Y^{1/2}$ in (\ref{eq:24}) can be computed from its Cholesky decomposition.

The expression in (\ref{eq:24}) gives $NM$ inequality constraints for the optimization problem. In summary, Problem 1 is converted into a convex programming problem.\\[-5pt]

\noindent
\textit{Problem 2}~\label{problem:2}: Given the system (\ref{eq:14}) and (\ref{eq:15}), find the optimal control sequences $V^*$ and $K^*$ that minimize the cost function (\ref{eq:16}) 
subject to the terminal state constraints (\ref{eq:17a}) and (\ref{eq:18}), and the \textit{individual} chance constraints (\ref{eq:24}).\\[-5pt]

\noindent
\textit{Remark 4}: Note that it is not possible to decouple the mean and covariance controllers in the presence of chance constraints, because of (\ref{eq:24}). 
%---------------------------------------------------%
\subsection{Risk Allocation Optimization}

Since $\delta_k^j$ are decision variables in (\ref{eq:24}), the constraints are \textit{bilinear}, which makes it difficult to solve this problem. 
As mentioned previously, in order to transform Problem~2 to a more tractable form, the allocation of the risk levels $\delta_k^j$ may be assumed to be fixed to some pre-specified values, usually uniformly. 
In this case, $\delta_k^j$ are no longer decision variables and the problem can be efficiently solved as an SDP. 
However, a better approach is to allocate $\delta_k^j$ concurrently when solving the optimization Problem~2, so as to minimize the total cost. This gives rise to a natural two-stage optimization framework~\cite{b1}.

Following the approach in \cite{b1}, the upper stage optimization finds the optimal risk allocation $\delta \coloneqq [\delta_1^1,\delta_1^2,\ldots,\delta_N^{M-1},\delta_N^{M}]\in\mathbb{R}^{N\!M}$, and the lower stage solves the CS problem for the optimal controller $U^* = U^*_{N-1}$ given the risk allocation $\delta$ from the upper-stage. 

Let the value of the objective function after the lower-stage optimization for a given risk allocation $\delta$ be $J^*$, that is,
\begin{equation}~\label{eq:25}
	J^*(\delta) = \min_{V,K} J(V,K),
\end{equation}
where $J(V,K)$ is given in (\ref{eq:16}). The upper-stage optimization problem can then be formulated as follows.
\\ \\
\noindent
\textit{Problem 3}~\label{problem:3} (Risk Allocation):
\begin{align}
	&\min_{\delta} J^*(\delta), \label{eq:26}\\
	\text{such~that} & \quad \sum_{k=1}^{N}\sum_{j=1}^{M} \delta_k^j \leq \Delta, \label{eq:27}\\
	&\qquad\delta_k^j > 0,  \label{eq:28}
\end{align}
As shown in \cite{b1}, Problem~3 is a convex optimization problem, given that the objective function $J(V,K)$ is convex, and $\Delta \in (0,0.5]$.
\subsection{Iterative Risk Allocation Motivation}
Even though we have formulated the solution of Problem~2 as a two-stage optimization problem, it is not clear yet how to solve Problem 3 efficiently in order to determine the optimal risk allocation. 
To gain insight into the solution, we first state a theorem from \cite{b1} about the monotonicity of $J^*(\delta)$. 

\begin{theorem} \label{theo:1}
	The optimal cost from solving Problem~2 is a monotonically \textit{decreasing} function in $\delta_k^j$, that is,
	\begin{equation}~\label{eq:29}
		\frac{\partial J^*}{\partial\delta_k^j} \leq 0,\quad  k=1,\ldots,N,~j=1,\ldots,M.
	\end{equation}
\end{theorem}
\noindent

\begin{proof}
	See \cite{b1}.
\end{proof}

In the following section, we use this theorem to create an iterative algorithm that solves Problem~3 by lowering the cost at each iteration. 
The main idea is that, by carefully increasing the risk allocations $\delta_{k}^{j}$, Theorem 1 guarantees that the optimal cost will be reduced with each successive iteration.
As such, the algorithm lowers the risk, equivalently tightens the constraints, for the constraints that are too conservative, and increases the risk, equivalently loosens the constraints, for the constraints that are already active.
The following remark introduces this idea and defines active and inactive constraints in the context of risk allocation.

\noindent
\\
\textit{Remark 5}: The chance constraints can be written in yet another form that will prove useful below. 
Starting from (\ref{eq:24}), notice that we can write the chance constraints as 
\begin{equation}~\label{eq:32}
	\delta_k^j \geq 1 - \Phi\Bigg(\frac{\beta_j-\alpha_j^\intercal E_k\bar{X}^*}{\|\Sigma_Y^{1/2}(I+\mathcal{B}K^*)^\intercal E_k^\intercal\alpha_j\|}\Bigg) \eqqcolon \bar{\delta}_k^j.
\end{equation}
The quantity $\bar{\delta}_k^j$ represents the \textit{true} risk experienced by the optimal trajectories, i.e, when using $(V^*,K^*)$. 
Clearly, the risk we have selected does not need to be equal to the actual risk once the optimization is completed. 
When these values are equal we will say that the constraint (\ref{eq:32}) is active, and is inactive otherwise. 
Good solutions  correspond to cases when the true risk is within a small margin of the allocated risk.
Many $\delta_k^j$ having values smaller than their true counterparts would imply an overly conservative solution.

\subsection{Iterative Risk Allocation Algorithm}

We can exploit Theorem \ref{theo:1} in the context of CS to create an iterative risk allocation algorithm that simultaneously finds the optimal risk allocation
$\delta^*$ and the optimal control pair $(V^*,K^*)$. 
To this end, suppose we start with some feasible risk allocation $\delta_{k(i)}^j$, for all $k,j$, where $i$ denotes the iteration number. 
Using this risk allocation, we then solve Problem~2 to get the optimal controller $(V_{(i)}^*,K_{(i)}^*)$, which corresponds to the optimal mean trajectory $\bar{X}_{(i)}^*$ at iteration $i$.
Next, we construct a new risk allocation $\delta_{k(i)}^{j\prime}$ as follows: for all $k,j$ such that $\delta_{k(i)}^j$ is active, we keep the corresponding allocation the same, i.e, $\delta_{k(i)}^{j\prime} = \delta_{k(i)}^j$. 
However, for all $k,j$ such that $\delta_{k(i)}^j$ is inactive we let $\delta_{k(i)}^{j\prime} < \delta_{k(i)}^j$.
This corresponds to tightening the constraints. 
Since this new risk allocation is smaller, and since $\Phi^{-1}(z)$ is a monotonically increasing function, it follows that $\Phi^{-1}(1-\delta_{k(i)}^{j\prime}) > \Phi^{-1}(1-\delta_{k(i)}^j)$. 
Furthermore, this implies that 
\begin{align}~\label{eq:33}
	\alpha_j^\intercal E_k\bar{X}_{(i)}^* &< \beta_j - \|\Sigma_Y^{1/2}(I+\mathcal{B}K_{(i)}^*)^\intercal E_k^\intercal\alpha_j\|\Phi^{-1}(1-\delta_{k(i)}^{j\prime})\nonumber\\
	&< \beta_j - \|\Sigma_Y^{1/2}(I+\mathcal{B}K_{(i)}^*)^\intercal E_k^\intercal\alpha_j\|\Phi^{-1}(1-\delta_{k(i)}^{j}).
\end{align}
Constraint (\ref{eq:33}) ensures that the optimal solution for $\delta_{(i)}$ is feasible for $\delta_{(i)}^\prime$. 
Furthermore, since $\delta_{k(i)}^{j\prime} < \delta_{k(i)}^{j}$, it follows that $\mathcal{R}(\delta^\prime) \subseteq \mathcal{R}(\delta)$, so the optimal solution for $\delta_{(i)}$ is also the optimal solution for $\delta_{(i)}^\prime$ as well, hence $J^*(\delta^\prime) = J^*(\delta)$.

Next, we construct a new risk allocation $\delta_{k(i+1)}^j$ from $\delta_{k(i)}^{j\prime}$ as follows. 
For all $k,j$ such that $\delta_{k(i)}^{j\prime}$ is inactive, leave the new risk allocation the same. 
For all $k,j$ such that $\delta_{k(i)}^{j\prime}$ is active, let $\delta_{k(i+1)}^j > \delta_{k(i)}^{j\prime}$, which corresponds to relaxing the constraints. 
Following the same logic, 
Theorem \ref{theo:1} implies that $J^*(\delta_{(i)}^\prime) \geq J^*(\delta_{(i+1)})$. Thus, we have laid out an iterative scheme for a sequence of risk allocations $\{\delta_{(0)},\delta_{(1)},\ldots,\delta_{(i)}\}$ that continually lowers the optimal cost. 

This leads to Algorithm~\ref{algo:1} that solves the optimal risk allocation for the CS problem subject to chance constraints.  
Note that the algorithm is initialized with a constant risk allocation. To tighten the inactive constraints in Line 9, the corresponding risk is scaled by a parameter  $0<\rho<1$ that weighs the current risk with the \textit{true} risk from that solution. 
Additionally, to loosen the active constraints in Line 13, the corresponding risk is increased proportionally to the residual risk remaining. 

\begin{algorithm}[]~\label{algo:1}
	\KwIn{$\delta_k^j\gets\Delta/(NM),\epsilon,\rho$}
	\KwOut{$\delta^*,J^*,V^*,K^*$}
	\While{$|J^*-J^*_{\rm prev}|>\epsilon$}{
		$J^*_{\rm prev}\gets J^*$\\
		Solve Problem 2 with current $\delta$ to obtain $\bar{\delta}$\\
		$\hat{N}\gets$ number of indices where constraint is active\\
		\If{$\hat{N} = 0$ \normalfont{or} $\hat{N} = MN$}{
			\textbf{break}\;
		}
		\ForEach{$(k,j)$ \normalfont{such that} $j$\normalfont{th constraint at} $k$\normalfont{th time step is \textit{inactive}}}{
			$\delta_k^j \gets \rho \delta_k^j + (1-\rho)\bar{\delta}_k^j$
		}
		$\delta_{\text{res}}\gets\Delta - \sum_{k=1}^{N}\sum_{j=1}^{M} \delta_k^j$\\
		\ForEach{$(k,j)$ \normalfont{such that} $j$\normalfont{th constraint at} $k$\normalfont{th time step is \textit{active}}}{
			$\delta_k^j \gets \delta_k^j + \delta_{\text{res}}/\hat{N}$
		}
	}
	\caption{Iterative Risk Allocation CS}
\end{algorithm}

% % % % % % % % % % % % % % % % % % % % % % % % % 

\section{Cone Chance Constraints}~\label{sec:quad}

In many engineering applications polytopic constraints such as (\ref{eq:19}) are not realistic. 
Most often, the constraints have the form of a convex cone, namely, the feasible region is characterized by
\begin{equation}~\label{eq:34}
	\mathcal{X}^c := \{x\in\mathbb{R}^n : \|Ax+b\|_2 \leq c^\intercal x + d\}.
\end{equation}
Cone constraints such as (\ref{eq:34}) are more realistic, as they better describe the feasible space.  
As with the case of a polyhedral feasible state space $\mathcal{X}^p$, we want the state to be inside 
$\mathcal{X}^c$ throughout the whole time horizon. 
However, since the dynamics are stochastic, this assumption is relaxed to the condition that the 
probability that the state is not inside this set is less than or equal to $\Delta$.
In the context of convex cone state constraints, this condition becomes
\begin{subequations}~\label{eq:35}
	\begin{align}
		\!\mathbb{P}(\|Ax_k + b\|_2 \leq c^\intercal &x_k + d) \geq 1 - \delta_k, ~~ k=1,\ldots,N, \label{eq:35a}\\
		\sum_{k=1}^{N}\delta_k &\leq \Delta. \label{eq:35b}
	\end{align}
\end{subequations}

\noindent
\textit{Remark 6}: Although the set $\mathcal{X}^c$ is convex, the chance constraint
$\mathbb{P}(x \in \mathcal{X}^c) \ge 1 - \delta$ may not be convex. 
Specifically, for large $\delta_k$, it is possible that the chance constraint (\ref{eq:35a}) is non-convex~\cite{TwoSide}. 

Since there is no guarantee that (\ref{eq:35}) will be a convex constraint, we need to make a convex approximation so that (\ref{eq:35}) holds for all $\Delta \in (0,0.5]$. 

\subsection{Two-Sided Approximation of Cone Constraints}~\label{TS}
\begin{comment}
Recent work on two-sided affine chance constraints~\cite{TwoSide} has shown how to relax a general class of \textit{quadratic} constraints of the form
\begin{equation}~\label{eq:36}
\mathbb{P}((\hat{a}^\intercal \xi + \hat{b})^2 + (\hat{c}^\intercal \xi + \hat{d})^2 \leq \kappa) \geq 1 - \epsilon,
\end{equation}
where $\hat{a}, \hat{c}$ are real vectors of appropriate dimension, $\hat{b},\hat{d} \in \Re$, $\kappa > 0$,
where $\xi$ is a Gaussian random variable, and $\epsilon\in (0,0.5]$.
In~\cite{TwoSide} the authors proved that (\ref{eq:36}) can be conservatively approximated by the following convex constraints
\begin{subequations}~\label{eq:37}
\begin{align}
\mathbb{P}(|\hat{a}^\intercal \xi + \hat{b}| \leq f_1) &\geq 1 - \beta\epsilon, \label{eq:37a}\\
\mathbb{P}(|\hat{c}^\intercal \xi + \hat{d}| \leq f_2) &\geq 1 - (1 - \beta)\epsilon, \label{eq:37b}\\
f_1^2 + f_2^2 \leq \kappa,\qquad & \label{eq:37c}
\end{align}
\end{subequations}
where $\beta\in(0,1)$ represents a constant that balances the trade-off between violating any of the two constraints (\ref{eq:37a})-(\ref{eq:37b}).
\end{comment}
In order to approximate the cone chance constraint (\ref{eq:35a}), we first replace the cone constraint in (\ref{eq:35a})
with the quadratic chance constraint
%to $c^\intercal \bar{x}_k + d$ on the right hand side because the term $\kappa$ in (\ref{eq:36}) is a constant, i.e., the chance constraints become
\begin{equation} \label{eq:relaxed_cone_constraints}
	\mathbb{P}(\|Ax_k + b\|_2 \leq c^\intercal \bar{x}_k + d) \geq 1 - \delta_k, ~~ k = 1,\ldots,N.
\end{equation}

\textit{Remark 7}: The chance constraint (\ref{eq:relaxed_cone_constraints})  is a relaxation of the original chance constraint (\ref{eq:35a}). 
The proof of this result is given in Appendix~A.

Note that (\ref{eq:relaxed_cone_constraints}) can be equivalently written as
\begin{equation}~\label{eq:relaxed_cone_constraints_2}
	\mathbb{P}\left[\left( \sum_{i=1}^{n} (a_i^\intercal x_k + b_i)^2 \right)^{1/2} \leq \kappa_k \right] \geq 1 - \delta_k, ~~ k = 1,\ldots,N,
\end{equation}
where $\kappa_k := c^\intercal \bar{x}_k + d$ and $a_i^\intercal$ denotes the $i$th row of $A$.
Squaring both sides of (\ref{eq:relaxed_cone_constraints_2}) and letting $\psi_{i,k} := a_i^\intercal x_k + b_i$ yields 
\begin{equation}\label{eq:relaxed_cone_constraints_3}
	\mathbb{P}\left(\sum_{i=1}^{n}\psi_{i,k}^2 \leq \kappa_k^2\right) \geq 1 - \delta_k, ~~ k = 1,\ldots, N.
\end{equation}
The following proposition enables the conversion of a quadratic constraint of the form (\ref{eq:relaxed_cone_constraints_3})  to a collection of two-sided (absolute value) constraints.
To simplify the notation, below we drop the subscript $k$ from the corresponding expressions.

\begin{prop}  \label{Proposition1} 
	The quadratic constraint
	\begin{equation}\label{eq:relaxed_cone_constraints_30}
		\mathbb{P}\left(\sum_{i=1}^{n}\psi_{i}^2 \leq \kappa^2\right) \geq 1 - \delta,
	\end{equation}
	is satisfied if  the following constraints are satisfied
	\begin{subequations}~\label{eq:two_sided_constraints}
		\begin{align}
			&\hspace*{-0.2cm}\mathbb{P}(|\psi_{i}| \leq f_{i}) \geq 1 - \beta_i \delta, \quad i = 1,\ldots,n,  \label{eq:two_sided_constraints_1} \\
			&\sum_{i=1}^{n} f_{i}^2 \leq \kappa^2,  \label{eq:two_sided_constraints_2} \\
			&\sum_{i=1}^{n} \beta_i = 1, \label{eq:two_sided_constraints_3}
		\end{align}
	\end{subequations}
	for some non-negative $f_1,f_2,\ldots,f_n$ and $\beta_1,\beta_2,\ldots,\beta_n$.
\end{prop}
\begin{proof}
	Since $|\psi_i| \le f_i$ for all $i=1,\ldots,n$ implies that 
	$\sum_{i=1}^{n} \psi_{i}^2 \leq \sum_{i=1}^{n} f_{i}^2 \leq \kappa^2$, it follows that
	\begin{equation}~\label{eq:two_sided_proof_1}
		\mathbb{P}\left(\sum_{i=1}^{n} \psi_{i}^2 \leq \kappa^2 \right) \geq \mathbb{P}
		\left(\bigwedge_{i=1}^{n} | \psi_{i} |  \leq f_{i} \right). 
	\end{equation}
	From the reverse union bound (see Appendix~B), we have 
	\begin{subequations}
		\begin{align}
			\mathbb{P}\left(\bigwedge_{i=1}^{n} |\psi_{i} | \leq f_{i}\right) &\geq \sum_{i=1}^{n} \mathbb{P}( |\psi_{i}| \leq f_i ) - (n-1) \\
			&\geq \sum_{i=1}^{n} (1 - \beta_i \delta) - (n - 1) \\
			&= 1 - \delta \sum_{i=1}^{n} \beta_i = 1 - \delta. \label{eq:two_sided_proof_2}
		\end{align}
	\end{subequations}
	Lastly, combining (\ref{eq:two_sided_proof_1}) and (\ref{eq:two_sided_proof_2}) gives (\ref{eq:relaxed_cone_constraints_3}), which concludes the proof.
\end{proof}

We now present two methods for approximating the two-sided chance constraints (\ref{eq:two_sided_constraints_1}), one that utilizes a result from \cite{TwoSide}, and one that utilizes the reverse union bound. 

\subsubsection{Three-cut Outer Approximation}

We state, without proof, the following lemma from \cite{TwoSide}. 

\begin{lem} \label{lemma1} 
	Let $\xi\sim\mathcal{N}(\mu,\Sigma)$ be a jointly distributed Gaussian random vector with mean $\mu$ and positive definite covariance matrix $\Sigma$, and let $\delta\in (0,0.5]$.
	Let $LL^\intercal = \Sigma$ be the Cholesky decomposition of $\Sigma$. 
	Let $a,b\in\mathbb{R}$ and $\eta \in\mathbb{R}^{n}$ be decision variables.
	Then, 
	\begin{subequations}~\label{eq:three_cut_approximation_general}
		\begin{align}
			t &\geq \|L^\intercal \delta \|_2, \\
			a - \mu^\intercal \eta &\leq \Phi^{-1}(\delta) t, \\
			b - \mu^\intercal \eta  &\geq \Phi^{-1}(1-\delta) t, \\
			a - b &\leq 2\Phi^{-1}(\delta/2)t,
		\end{align}
	\end{subequations}
	is a second order cone (SOC) outer approximation of the constraint
	\begin{equation}
		\mathbb{P}(a \leq \eta^\intercal \xi \leq b) \geq 1 - \delta, \label{eq:two_sided_xi}
	\end{equation}
	which, in fact, guarantees 
	\begin{equation}
		\mathbb{P}(a \leq \eta^\intercal \xi \leq b) \geq 1 - 1.25\delta.
	\end{equation}
\end{lem}

%In order to use Lemma 1, we need to put the chance constraints (\ref{eq:two_sided_constraints_1}) into the form (\ref{eq:two_sided_xi}). 
%To do so, plug in the mean and deviation dynamics (\ref{eq:14})-(\ref{eq:15}) into $\psi_{i,k}$, which after rearranging terms, yields 
%\begin{equation}
%    \mathbb{P}(\hat{a} \leq \hat{x}^\intercal Y \leq \hat{b}) \geq 1 - \hat{\epsilon},
%\end{equation}
%where 
%\begin{subequations}
%    \begin{align}
%        \hat{a} &= -[f_{i} + b_{i} + a_{i}^\intercal E_k(\mathcal{A} \mu_0 + \mathcal{B} V)], \\
%        \hat{b} &= f_i - b_i - a_i^\intercal E_k (\mathcal{A} \mu_0 + \mathcal{B} V), \\
%        \hat{x} &= (I + \mathcal{B} K)^\intercal E_k^\intercal a_i, \\
%        \hat{\epsilon} &= \beta_i \delta_k. 
%    \end{align}
%\end{subequations}
%Note that $Y \sim \mathcal{N}(0,\Sigma_Y)$, where from before $\Sigma_Y = \mathcal{A}\Sigma_0\mathcal{A}^\intercal + \mathcal{D}\mathcal{D}^\intercal$. 
Using Lemma~\ref{lemma1} and Proposition~\ref{Proposition1}, it follows immediately that 
the constraint (\ref{eq:relaxed_cone_constraints}) is satisfied if the following 
convex constraints are satisfied, for all $i = 1,\ldots,n$ and $k=1,\ldots, N$
\begin{subequations}~\label{eq:three-cut-approx}
	\begin{align}
		t_{i,k} \geq \|\Sigma_{Y}^{1/2}(I + \mathcal{B} &K)^\intercal E_k^\intercal a_i \|_2, \\
		-[f_{i,k} + b_i + a_i^\intercal E_k(\mathcal{A}\mu_0 + \mathcal{B} V)] &\leq \Phi^{-1}(\beta_i\delta_k) t_{i,k}, \\
		f_{i,k} - b_i - a_i^\intercal  E_k(\mathcal{A}\mu_0 + \mathcal{B} V) &\geq \Phi^{-1}(1-\beta_i\delta_k) t_{i,k}, \\
		-f_{i,k} &\leq \Phi^{-1}(\beta_i\delta_k/2)t_{i,k},
	\end{align}
\end{subequations}
for the decision variables $V,K, \mathbf{t} := [\mathbf{t}_1,\ldots,\mathbf{t}_N]^\intercal$, and  $\mathbf{f} := [\mathbf{f}_1,\ldots,\mathbf{f}_N]$, where $\mathbf{t}_k := [t_{1,k},\ldots,t_{n,k}]$ and $\mathbf{f}_k := [f_{1,k},\ldots,f_{n,k}]$.

% % % % % % % % % % % % % % % % % % % % % % % % % % % % % % % % %

\subsubsection{Reverse Union Bound Approximation}

The following proposition gives an alternate approximation of the two-sided chance constraint (\ref{eq:two_sided_constraints_1}) using the cumulative distribution function of the normal distribution.
\begin{prop}
	Let $\epsilon^1_{i,k}, \epsilon^2_{i,k} > 0$ for all $i = 1,\ldots,n$ and $k = 1,\ldots,N$.
	Assume that 
	the convex SOC constraints
	\begin{subequations}~\label{eq:reverse_union_bound_constraints}
		\begin{align}
			&a_i^\intercal E_k(\mathcal{A} \mu_0 + \mathcal{B} V) + \|\Sigma_Y^{1/2}(I + \mathcal{B} K)^\intercal E_k^\intercal a_i\|_2 \Phi^{-1}(\epsilon^1_{i,k}) \nonumber \\
			&\hspace*{6cm} \leq f_{i,k} - b_i, \\
			&-a_i^\intercal E_k (\mathcal{A} \mu_0 + \mathcal{B} V) + \|\Sigma_{Y}^{1/2}(I + \mathcal{B} K)^\intercal E_k^\intercal a_i\|_2 \Phi^{-1}(\epsilon^2_{i,k})  \nonumber \\
			&\hspace*{6cm} \leq f_{i,k} + b_i, 
		\end{align}
	\end{subequations}
	are satisfied for some $\epsilon^1_{i,k} + \epsilon^2_{i,k} \geq 2 - \beta_i\delta_k$, $V$ and $K$.
	Then, the  chance constraints (\ref{eq:two_sided_constraints_1}) are satisfied as well.
\end{prop}

\begin{proof}
	Equation (\ref{eq:two_sided_constraints_1}) can equivalently be written as 
	\begin{align}
		\mathbb{P} (-f_{i,k}-b_i \leq a_i^\intercal x_k \leq f_{i,k} - b_i) \geq 1 - \beta_i \delta_k,
	\end{align}   
	or, equivalently, as
	\begin{align}    \label{eq:decoupled_two_sided}
		\mathbb{P}(a_i^\intercal x_k \leq f_{i,k} - b_i \wedge a_i^\intercal x_k \geq - f_{i,k} - b_i) \geq 1 - \beta_i \delta_k. 
	\end{align}
	Applying the reverse union bound on (\ref{eq:decoupled_two_sided}) yields
	\begin{align}
		&\mathbb{P}(a_i^\intercal x_k \leq f_{i,k} - b_i \wedge a_i^\intercal x_k \geq - f_{i,k} - b_i) \nonumber \\
		&\geq \mathbb{P}(a_i^\intercal x_k \leq f_{i,k} - b_i) + \mathbb{P}(a_i^\intercal x_k \geq - f_{i,k} - b_i) - 1.\label{eq:reverse_union_bound}
	\end{align}
	Thus, if the constraint
	\begin{equation}
		\mathbb{P}(a_i^\intercal x_k \leq f_{i,k} - b_i) + \mathbb{P}(a_i^\intercal x_k \geq - f_{i,k} - b_i) \geq 2 - \beta_i \delta_k, \label{eq:sum_of_probabilities}
	\end{equation}
	is satisfied, then inequality (\ref{eq:decoupled_two_sided}) holds and
	the original chance constraints (\ref{eq:two_sided_constraints_1}) will be satisfied as well. 
	The constraint in (\ref{eq:sum_of_probabilities}) is equivalent to the following decoupled constraints
	\begin{subequations}
		\begin{align}
			&\mathbb{P}(a_i^\intercal x_k \leq f_{i,k} - b_i) \geq \epsilon^1_{i,k}, \label{eq:regular_CC_1} \\
			&\mathbb{P}(a_i^\intercal x_k \geq - f_{i,k} - b_i) \geq \epsilon^2_{i,k}, \label{eq:regular_CC_2} \\
			&\epsilon^2_{i,k} + \epsilon^2_{i,k} \geq 2 - \beta_i \delta_k.
		\end{align}
	\end{subequations}
	Since $a_i^\intercal x_k$ is a Gaussian random variable with mean $a_i^\intercal \bar{x}_k$ and covariance $a_i^\intercal E_k \Sigma_{X} E_k^\intercal a_i$, it follows that the probabilities in (\ref{eq:regular_CC_1})-(\ref{eq:regular_CC_2}) can be 
	written in terms of the standard normal cumulative distribution function as follows
	\begin{subequations}~\label{eq:CDF_decoupled_CC}
		\begin{align}
			&\Phi\left[\frac{f_{i,k} - b_i - a_i^\intercal \bar{x}_k}{\sqrt{a_i^\intercal E_k \Sigma_X E_k^\intercal a_i}}\right] \geq \epsilon^1_{i,k}, \\
			&\Phi\left[\frac{f_{i,k} + b_i + a_i^\intercal \bar{x}_k}{\sqrt{a_i^\intercal E_k \Sigma_{X} E_k^\intercal a_i}}\right] \geq \epsilon^2_{i,k}.
		\end{align}
	\end{subequations}
	Finally, substituting the mean dynamics (\ref{eq:14}) and the corresponding covariance in to (\ref{eq:CDF_decoupled_CC}) yields the desired result.
\end{proof}

\noindent
\textit{Remark 8}: In the three-cut approximation, the parameters $\mathbf{\beta} := [\beta_1,\ldots,\beta_n]^\intercal$ need to be fixed so that the resulting constraints in (\ref{eq:three-cut-approx}) are not bilinear in the decision variables.
Similarly, in the reverse union bound approximation, the parameters $\epsilon^1_{i,k}$ and $\epsilon^2_{i,k}$ need to be fixed so that the resulting constraints in (\ref{eq:reverse_union_bound_constraints}) are not bilinear in the decision variables. 
\\

\noindent
\textit{Remark 9}: In (\ref{eq:two_sided_constraints_2}), the constraints can be equivalently written as 
\begin{equation}
	\| \mathbf{f}_k\|_2 \leq \kappa_k, ~~ k = 1,\ldots, N.
\end{equation}
Plugging in the mean dynamics (\ref{eq:14}) into the definition of $\kappa_k$ yields the SOC constraints
\begin{equation}
	\|\mathbf{f}_k\|_2 \leq c^\intercal E_k(\mathcal{A}\mu_0 + \mathcal{B} V) + d,
	\quad k=1,\ldots,N.
	\label{eq:SOC_f}
\end{equation}

In summary, the relaxed cone chance constraints (\ref{eq:relaxed_cone_constraints}) can be approximated using two-sided chance constraints as in (\ref{eq:three-cut-approx}) with (\ref{eq:SOC_f}) through the three-cut approximation, or alternatively as in (\ref{eq:reverse_union_bound_constraints}) with (\ref{eq:SOC_f}) through the reverse union bound inequality.
Since these constraints are convex, the resulting problem is convex and can be solved using standard SDP solvers, similarly to the polyhedral chance constraint case.
In the three-cut approximation, there will a total of $4n + 1$ constraints per time step, or $(4n+1)N$ total SOC constraints.
In the reverse union bound approximation, there will be a total of $2n + 1$ constraints per time step, or $(2n+1)N$ total SOC constraints.
Although the RUB relaxation is computationally more efficient, the three-cut relaxation is, in general, less conservative.
This was confirmed from our numerical examples in Section~\ref{sec:NumEx}.

\subsection{Geometric Approximation}

We limit the following discussion to the three-dimensional case,  which is
often the case in practice when enforcing position constraints.
However, the results can be generalized to $n$-dimensional convex cones by using a concentration inequality for $\chi^2$ random variables \cite{JoshJack}. 
For simplicity, let $b = 0$ in (\ref{eq:34}), which corresponds to a cone centered at the origin. 
Letting the state be $x := [p^\intercal,\dot{p}^\intercal]^\intercal$, where $p\in\mathbb{R}^3$ denotes the position, the chance constraints (\ref{eq:35}) become
\begin{align}
	&\mathbb{P}\left(\bigg\|
	\begin{bmatrix}
		A_c & 0 \\
		0 & 0 
	\end{bmatrix}
	\begin{bmatrix}
		p_k \\
		\dot{p}_k
	\end{bmatrix}\bigg\| \leq [c_c \ 0]^\intercal 
	\begin{bmatrix}
		p_k \\
		\dot{p}_k
	\end{bmatrix}
	+ d\right) \geq 1 - \delta_k,
\end{align}
or equivalently,
\begin{align} \label{chance:eq45}
	\mathbb{P}(\|A_c p_k \| \leq c_c^\intercal p_k + d) \geq 1 - \delta_k,
\end{align}
where $c_c\in\mathbb{R}^{3}, A_c \in \mathbb{R}^{3\times 3}$ parametrize the cone.
In most three-dimensional applications, the matrix $A_c$ has two nonzero diagonal elements, and one zero diagonal element. 
As such, the vector $A_c p_k$ will have one zero element.
Let $H\in\mathbb{R}^{2\times 3}$ be defined such that it extracts the nonzero elements of $A_c p_k$.
This is needed so as to reduce the dimensionality of the random vector inside the norm, from which we can more easily approximate the chance constraint.
It follows that the chance constraints (\ref{chance:eq45}) become
\begin{equation}
	\mathbb{P}(\|H A_c I_p E_k X\| \leq c_c^\intercal I_p E_k X + d) \geq 1 - \delta_k, \label{eq:geometric_cone_CC}
\end{equation}
where $I_p := [I_{3},0_{3}]\in\mathbb{R}^{3\times 6}$.
From a geometric point of view, one can think of the constraints (\ref{eq:geometric_cone_CC}) as imposing, at each time step $k$, 
that the random vector $\xi_k := HA_cp_k \in \mathbb{R}^2$ lies inside the \textit{disk} of radius $r_k = c_c^\intercal p_k + d$ with probability greater than $1 - \delta_k$, 
However, since $X$ is a stochastic process, it follows that the radius of the disk is uncertain, therefore, and similar to Section~\ref{TS}, we relax the chance constraint such that the Gaussian vector $\xi$ lies within the \textit{mean} radius of the disk $\bar{r}_k = c_c^\intercal I_p E_k \bar{X} + d$.

Using this approximation, the chance constraints (\ref{eq:35a})  become
\begin{equation}\label{eq:I}
	\mathbb{P}(\|\xi_k\|_2 \leq \bar{r}_k) \geq 1 - \delta_k.
\end{equation}
Note that the random variable $\xi_k = AE_k X$ is Gaussian such that $\xi_k \sim \mathcal{N}(\bar{\xi}_k,\Sigma_{\xi_k})$, with mean $\bar{\xi}_k := HA_cI_pE_k\bar{X}$ and covariance $\Sigma_{\xi_k} := HA_cI_pE_k\Sigma_XE_k^\intercal I_p^\intercal A_c^\intercal H^\intercal$. 
So far, we have turned the convex cone chance constraint (\ref{eq:35a}) into the chance constraint (\ref{eq:I}) that requires the probability of a Gaussian random vector 
being inside a circle of a given radius be greater than $1 - \delta_k$. 
Similar to the methodology in \cite{JoshJack}, this problem can be analytically solved as follows.
\begin{prop}
	Let $\zeta \sim \mathcal{N}(0,\Sigma_{\zeta})$ be a two-dimensional random vector. 
	Then, for $a>0$,
	\begin{equation}\label{eq:III}
		\mathbb{P}\big(\zeta^\intercal \Sigma_{\zeta}^{-1}\zeta \leq a^2\big) = 1 - e^{-\frac{1}{2}a^2}.
	\end{equation}
\end{prop}

\begin{proof}
	The probability density function (PDF) of $\zeta$ is given by
	\begin{equation}
		\mathcal{N}(0,\Sigma_{\zeta}) = \frac{1}{2\pi|\det\Sigma_{\zeta}|^{\frac{1}{2}}}e^{-\frac{1}{2}\zeta^\intercal\Sigma_{\zeta}^{-1}\zeta}.
	\end{equation}
	Then, the probability in (\ref{eq:III}) is given explicitly by
	\begin{equation}\label{eq:IV}
		\mathbb{P}(\zeta^\intercal\Sigma_{\zeta}^{-1}\zeta \leq a^2) = \frac{1}{2\pi|\det\Sigma_{\zeta}|^{\frac{1}{2}}}\int_{\Omega_{\zeta}} e^{-\frac{1}{2}\zeta^\intercal\Sigma_{\zeta}^{-1}\zeta}\, \d\zeta,
	\end{equation}
	where $\Omega_{\zeta} := \{\zeta: \zeta^\intercal \Sigma_{\zeta}^{-1} \zeta \leq a^2\}$.
	%	Physically, this corresponds to integrating over an \textit{ellipsoidal} area in two dimensions, where the ellipse is characterized by the matrix $\Sigma_{\zeta}^{-1}$. 
	Changing coordinates such that $\nu := \Sigma_{\zeta}^{-\frac{1}{2}}\zeta = (\rho\cos\phi,\rho\sin\phi)$ so that $\d\nu = |\det \Sigma_{\zeta}|^{-\frac{1}{2}} \, \d\zeta$, note that the sets $\{\zeta^\intercal\Sigma_{\zeta}^{-1}\zeta \leq a^2\}$ and $\{\|\nu\|_2 \leq a\}$ are equivalent. Thus, the integral in (\ref{eq:IV}) becomes
	\begin{equation}
		\mathbb{P}(\zeta^\intercal\Sigma_{\zeta}^{-1}\zeta \leq a^2) = \mathbb{P}(\|\nu\|_2 \leq a) = \frac{1}{2\pi} \int_{\Omega_{\nu}}e^{-\frac{1}{2}\nu^\intercal \nu} \, \d\nu,
	\end{equation}
	where $\Omega_{\nu} := \{\nu: \|\nu\|_2 \leq a\}$.
	The last integral is straightforward to evaluate in two dimensions, namely, 
	\begin{equation}~\label{eq:V}
		\mathbb{P}(\|\nu\|_2 \leq a) = \frac{1}{2\pi}\int_{0}^{2\pi}\int_{0}^{a}e^{-\frac{1}{2}\rho^2}  \rho \, \d \rho\, \d\phi = 1 - e^{-\frac{1}{2}a^2},
	\end{equation}
	which yields the desired result.
\end{proof}

\begin{lem}~\label{lemma:1}
	Let $\zeta \sim \mathcal{N}(0,\Sigma_{\zeta})$ be a two-dimensional random vector, let $\sigma_{\zeta}^2 = \lambda_{\mathrm{max}}(\Sigma_{\zeta})$, and let $r > 0$. 
	Then,
	\begin{equation}~\label{eq:VI}
		\mathbb{P}(\|\zeta\|_2 \leq r) \geq 1 - e^{-r^2/2\sigma_{\zeta}^2}.
	\end{equation}
\end{lem}

\begin{proof}
	Since the covariance matrix is positive definite, we can decompose it as $\Sigma_{\zeta} = PDP^\intercal$ where $D$ is a diagonal matrix containing the eigenvalues $\lambda_i$ of $\Sigma_{\zeta}$ and $P$ is an orthogonal matrix.
	Since $\sigma_{\zeta}^2 = \max_i\lambda_i$, it follows that 
	\begin{equation}
		D^{-1} = \frac{1}{\sigma_{\zeta}^2}\text{diag}(\sigma_{\zeta}^2 / \lambda_i) \geq \frac{1}{\sigma_{\zeta}^2}I.
	\end{equation}
	From the previous expression, it follows that 
	\begin{equation}
		\zeta^\intercal \Sigma_{\zeta}^{-1} \zeta = \zeta^\intercal P D^{-1} P^\intercal \zeta \geq \frac{1}{\sigma_{\zeta}^2}\zeta^\intercal P P^\intercal \zeta = \frac{1}{\sigma_{\zeta}^2}\|\zeta\|_2^2.
	\end{equation}
	Rearranging the previous inequality gives $\|\zeta\|_2^2/\sigma_{\zeta}^2 \leq \zeta^\intercal \Sigma_{\zeta}^{-1}\zeta$, and using (\ref{eq:III}), it follows that 
	\begin{equation}
		\mathbb{P}(\|\zeta\|_2^2 \leq \sigma_{\zeta}^2 a^2) \geq \mathbb{P}(\zeta^\intercal \Sigma_{\zeta}^{-1} \zeta \leq a^2) = 1 - e^{-\frac{1}{2}a^2}.
	\end{equation}
	Setting $r^2 = \sigma_{\zeta}^2 a^2$ achieves the desired result. 
	Geometrically, the level sets $\{\zeta^\intercal \Sigma_{\zeta}^{-1} \zeta = a^2\}$ define the contours of ellipses having probability $1 - e^{-a^2/2}$ and the level sets $\{\|\zeta\|_2^2 = r^2\}$ are the smallest circles that contain these ellipses.
\end{proof}

\begin{theorem}~\label{prop:2}
	Let $\xi\sim\mathcal{N}(\bar{\xi},\Sigma_{\xi})$ be a two-dimensional random vector, let $\sigma_{\xi}^2 = \lambda_{\mathrm{max}}(\Sigma_{\xi})$, and let $r > 0$. 
	Then,
	\begin{equation}~\label{eq:Thm2}
		\|\bar{\xi}\|_2 + \sigma_{\xi}\sqrt{2\log\frac{1}{\delta}} \leq r \Rightarrow \mathbb{P}(\|\xi\|_2 \leq r) \geq 1 - \delta.
	\end{equation}
\end{theorem}

\begin{proof}
	First, note that for $\|\xi\|_2 \leq r$, the following implications hold
	\begin{subequations}~\label{eq:pfThm2}
		\begin{align}
			\|\bar{\xi}\|_2 + &\sigma_{\xi}\sqrt{2\log\frac{1}{\delta}} \leq r \Rightarrow \sigma_{\xi}\sqrt{2\log\frac{1}{\delta}} \leq r - \|\bar{\xi}\|\\
			&\Rightarrow 2\sigma_{\xi}^2\log\frac{1}{\delta} \leq (r - \|\bar{\xi}\|_2)^2\\
			& \Rightarrow  2\sigma_{\xi}^2\log\delta \geq -(r - \|\bar{\xi}\|_2)^2\\
			& \Rightarrow  \log\delta \geq - \frac{(r - \|\bar{\xi}\|_2)^2}{2\sigma_{\xi}^2}\\
			& \Rightarrow  \delta \geq \exp\bigg( - \frac{(r - \|\bar{\xi}\|_2)^2}{2\sigma_{\xi}^2}\bigg)\\
			& \Rightarrow  1 - \delta \leq 1 - \exp\bigg( - \frac{(r - \|\bar{\xi}\|_2)^2}{2\sigma_{\xi}^2}\bigg).  \label{eq:pfThm2f}
		\end{align}
	\end{subequations}
	%	From the triangle inequality, $\|\xi\|_2 = \|\bar{\xi} + \tilde{\xi}\|_2 \leq \|\bar{\xi}\|_2 + \|\tilde{\xi}\|_2$. 
	Since $\{\xi: \|\bar{\xi}\|_2 + \|\tilde{\xi}\|_2 \leq r\} \subseteq \{\xi: \|\xi\|_2 \leq r\}$, where $\tilde{\xi} := \xi - \bar{\xi}$, it follows that
	\begin{equation}\label{eq:triangle}
		\mathbb{P}(\|\xi\|_2 \leq r) \geq \mathbb{P}(\|\bar{\xi}\|_2 + \|\tilde{\xi}\|_2 \leq r) = \mathbb{P}(\|\tilde{\xi}\|_2 \leq r - \|\bar{\xi}\|_2).
	\end{equation}
	Since $\tilde{\xi}$ is a \textit{zero-mean} Gaussian vector, applying Lemma 2 gives
	\begin{equation}
		\mathbb{P}(\|\tilde{\xi}\|_2 \leq r - \|\bar{\xi}\|_2) \geq 1 - \exp\bigg(- \frac{(r - \|\bar{\xi}\|_2)^2}{2\sigma_{\xi}^2}\bigg).\\
	\end{equation}
	Finally, by (\ref{eq:pfThm2}) and (\ref{eq:triangle}), we obtain the desired result.
\end{proof}

Using Theorem~\ref{prop:2}, we can now satisfy (\ref{eq:I}) by enforcing
\begin{equation}\label{eq:VII}
	\sigma_{\xi_k}\sqrt{2\log\frac{1}{\delta_k}} \leq \bar{r}_k - \|\bar{\xi}_k\| =: \bar{R}_k.
\end{equation}
%Lastly, we get rid of the singular value term in (\ref{eq:VII}) for a more explicit formula in terms of the matrix $\Sigma_{\xi}$ from before. 
%Note that $\sigma_{\xi_k}^2 = \lambda_{\mathrm{max}}(\Sigma_{\xi}) = \lambda_{\mathrm{max}}(AE_k\Sigma_XE_k^\intercal A^\intercal)$.
%Also note that by definition, the induced two-norm of a matrix $B$ is $\|B\|_2 := \sqrt{\lambda_{\mathrm{max}}(B^\intercal B)}$. 
Using $\Sigma_X = (I + \mathcal{B}K)\Sigma_Y(I + \mathcal{B}K)^\intercal$ and noting that $\sigma_{\xi_k}^2 = \lambda_{\textrm{max}}(\Sigma_{\xi})$, we obtain
\begin{equation}
	\sigma_{\xi_k}^2 = 
	%\lambda_{\mathrm{max}}(AE_k\Sigma_XE_k^\intercal A^\intercal) = 
	\|\Sigma_Y^{1/2}(I + \mathcal{B}K)^\intercal E_k^\intercal I_p^\intercal A_c^\intercal H^\intercal\|_2^2.
\end{equation}
In summary, the cone chance constraints (\ref{eq:35a}) become 
\begin{align}~\label{eq:62}
	&\!\sqrt{2\log\frac{1}{\delta_k}}\|\Sigma_Y^{1/2}(I + \mathcal{B}K)^\intercal E_k^\intercal I_p^\intercal A_c^\intercal H^\intercal\|_2 \leq \bar{R}_k,
\end{align}
for $ k = 1,\ldots,N$.

\section{Spacecraft Rendezvous Example}  \label{sec:NumEx}

\subsection{IRA-CS with Polytopic Chance Constraints}

In this section, we implement the previous theory of CS with optimal risk allocation to the problem of spacecraft proximity operations in orbit. 
We consider the problem where one of the spacecraft, called the Deputy, approaches and docks with the second spacecraft, called the Chief, 
such that in the process, the Deputy remains within the line-of-sight (LOS) of the Chief, defined initially to be the polytopic region shown in Figure~\ref{fig:2}. 
%Afterwards, we will consider the more natural case of a cone region and compare the two problem formulations.
%
\begin{figure}[!htb]
	\centering
	\includegraphics[scale=0.38]{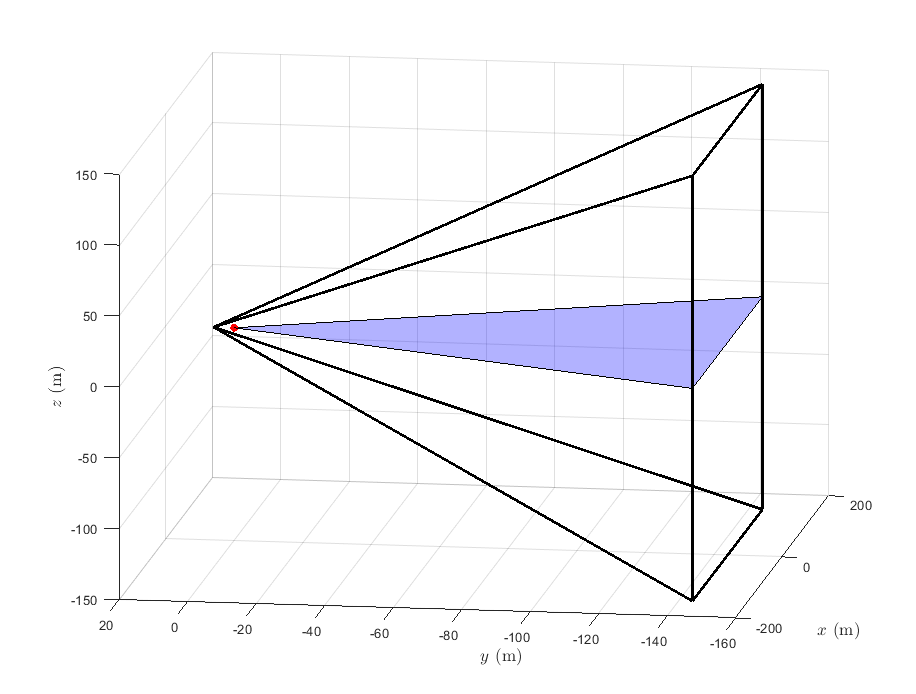}
	\caption{Feasible state space region for spacecraft rendezvous problem.}
	\label{fig:2}
\end{figure}

It is assumed that the two spacecraft are in the LVLH frame, that is, a rotating reference frame where the $z$ axis is oriented in the direction of the center of the Earth, the $y$ axis negative to the orbit normal, and the $x$ axis to complete the right hand rule.
Moreover, assuming that the Chief is in a circular orbit (at an altitude $h = 800$ km), the relative dynamics of the motion between the two spacecraft are given by the Clohessy-Wiltshire-Hill Equations~\cite{CWH},
\begin{subequations}~\label{eq:40}
	\begin{align}
		\ddot{x} &= 3\omega^2 x + 2\omega \dot{y} + F_x/m_d,\\
		\ddot{y} &= -2\omega\dot{x} + F_y/m_d,\\
		\ddot{z} &= -\omega^2 z + F_z/m_d,
	\end{align}
\end{subequations}
where $m_d = 300$ kg is the mass of the Deputy, $\omega = \sqrt{\mu/R_0^3}$ is the orbital frequency, and $F\coloneqq[F_x,F_y,F_z]^\intercal$ represents the thrust input components 
to the spacecraft. 
It is assumed that the thrust is generated by a chemical propulsion system with PWM (pulse-width-modulation) able to implement continuous thrust profiles from impulsive forces.

The equations of motion (\ref{eq:40}) are written in a relative coordinate system, where the Chief is located at the origin, and $x,y,z$ represent the position of the Deputy with respect to the Chief. Note that the $z$ dynamics are decoupled from the $x-y$ dynamics; furthermore, the $z$ dynamics are globally asymptotically stable, so in theory we only need to control the planar dynamics. 
In Figure \ref{fig:2} the blue area represents the planar region. 
To write the system in state space form, let $x \coloneqq [p_x,p_y,p_z,\dot{p}_x,\dot{p}_y,\dot{p}_z]^\intercal\in\mathbb{R}^6$ to obtain the LTI system $\dot{x} = Ax+Bu$, where
\begin{equation}~\label{eq:41}
	A = 
	\begin{bmatrix}
		0 & 0 & 0 & 1 & 0 & 0\\
		0 & 0 & 0 & 0 & 1 & 0\\
		0 & 0 & 0 & 0 & 0 & 1\\
		3\omega^2 & 0 & 0 & 0 & 2\omega & 0\\
		0 & 0 & 0 & -2\omega & 0 & 0\\
		0 & 0 & -\omega^2 & 0 & 0 & 0
	\end{bmatrix},
	\
	B = m_c^{-1}[0_3,I_3]^\intercal,
\end{equation}
and
$u \coloneqq [F_x,F_y,F_z]^\intercal\in\mathbb{R}^3$. 
To discretize the system, we divide the time interval into $N = 15$ steps, with a time interval $\Delta t = 4$~sec. 
Assuming a zero-order hold (ZOH) on the control and adding noise that captures any modeling and discretization errors, as well as other environmental disturbances, yields the discrete system
\begin{equation}
	x_{k+1} = A_dx_k + B_du_k + Gw_k,
\end{equation}
where $A_d = e^{A\Delta t}$, $B_d = \int_{0}^{\Delta t}e^{A\tau}B\ \d\tau$.
We choose the associated noise characteristics as $G = \text{diag}(10^{-4},10^{-4},5\times 10^{-8},5\times 10^{-8})$~\cite{ADF}. 
We assume that the initial state mean and covariance are $\mu_0 = [90,-120,90,0,0,0]^\intercal$ 
and $\Sigma_0 = \text{diag}(10,10,10,1,1,1)$,
respectively. 
We wish to steer the distribution from the above initial state to the final mean $\mu_f = 0$ with final 
covariance $\Sigma_f = \frac{1}{4}\Sigma_0$, while minimizing the cost function (\ref{eq:4}) with weight matrices $Q = \text{diag}(10,10,10,1,1,1)$ and $R = 10^3I_3$.
We impose the joint probability of failure over the whole horizon to be $\Delta = 0.03$, which implies that the probability of violating any state constraint over the whole horizon is less than 3\%. 
The control inputs are bounded as $\|u_k\|_{\infty} \leq 30$ N, which corresponds to a maximum acceleration of 10 cm/s$^2$.
Note that these bounds are \textit{hard} constraints as opposed to (soft) chance constraints. 
To implement this input hard constraint within the CS framework, the algorithm in \cite{IH} was used (see also Appendix~C).
It should be noted here that since saturation of the input may lead to non-Gaussian state evolution, the chance constraint inequality (\ref{eq:23}) may not hold anymore. 
For our purposes though, the formulated SOC constraints work well even for the non-Gaussian case.
This may lead to somewhat more conservative results, but for our problem, the difference turned out to be negligible.
%The details are given in the Appendix.

Lastly, in the iterative risk allocation algorithm, we use a scaling parameter $\rho_{(i)} = (0.7)(0.98)^i$ in Line~10 of the algorithm, where $i$ 
represents the current iteration. 
The SDP in Problem~2 was implemented in MATLAB using YALMIP\cite{YALMIP} along with MOSEK\cite{MOSEK} to solve the relevant optimization problems.

\begin{figure}[!htb]
	\centering
	\includegraphics[scale=0.32]{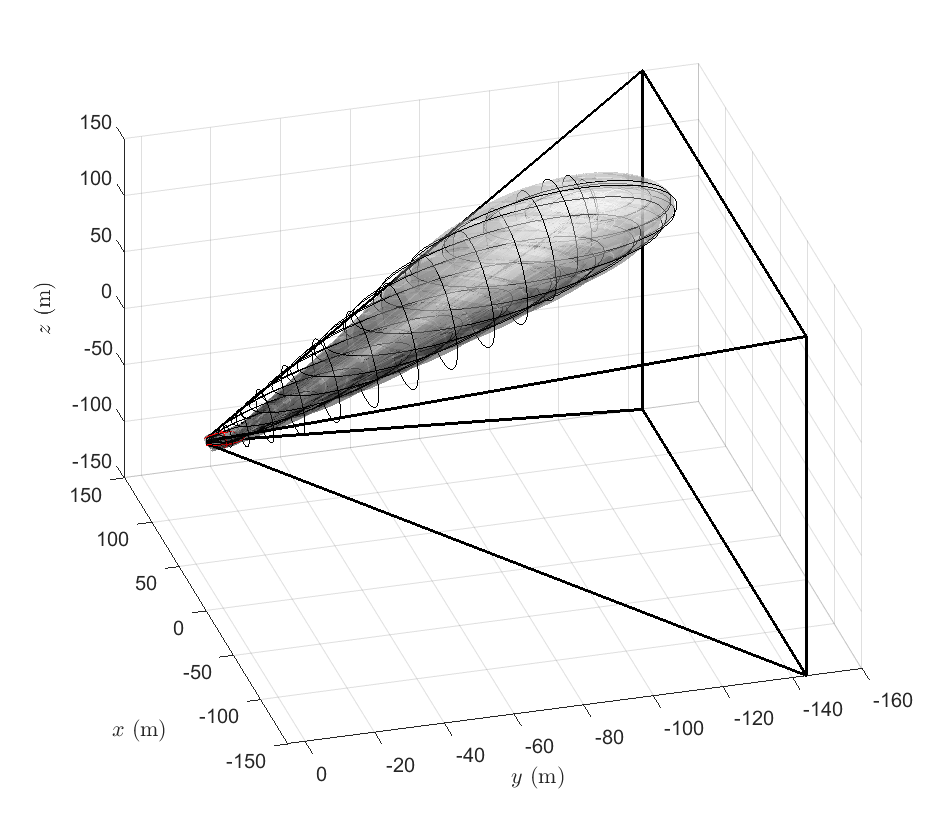}
	\caption{Optimal trajectories using IRA-CS, with $3-\sigma$ covariance ellipsoids.}
	\label{fig:3}
\end{figure}

\begin{figure}[!htb]
	\centering
	\hspace*{-0.5cm}
	\includegraphics[scale=0.45]{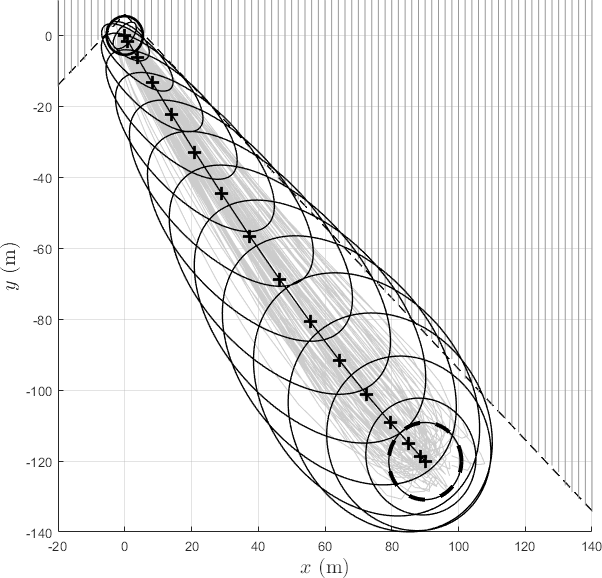}
	\caption{Optimal planar trajectories using IRA-CS, with 3-$\sigma$ covariance ellipses.}
	\label{fig:5}
\end{figure}

Figures~\ref{fig:3} and \ref{fig:5} show the trajectories with optimal risk allocation, and Figure~\ref{fig:5} shows the two-dimensional planar motion. 
Figure~\ref{fig:6} compares the terminal trajectories of CS with a uniform risk allocation with the proposed method. 
The two solutions look similar and both satisfy the terminal constraints on the mean and the covariance. 
However, due to the relaxation $\Sigma_N \leq \Sigma_f$, the uniform risk allocation scheme leads to more conservative solutions, as shown in Figure~\ref{fig:6}. 
The volume of the final covariance ellipsoid, $V_{N} \propto \log\det\Sigma_N$ is smaller for the uniform allocation solution compared to the optimal allocation solution (see Table~\ref{table:1}). 
In fact, we see that a consequence of optimal risk allocation is that it \textit{maximizes} the final covariance given all the constraints, while still being bounded by $\Sigma_f$.
\begin{figure}[!htb]
	\centering
	\hspace*{-0.5cm}
	\includegraphics[scale=0.27]{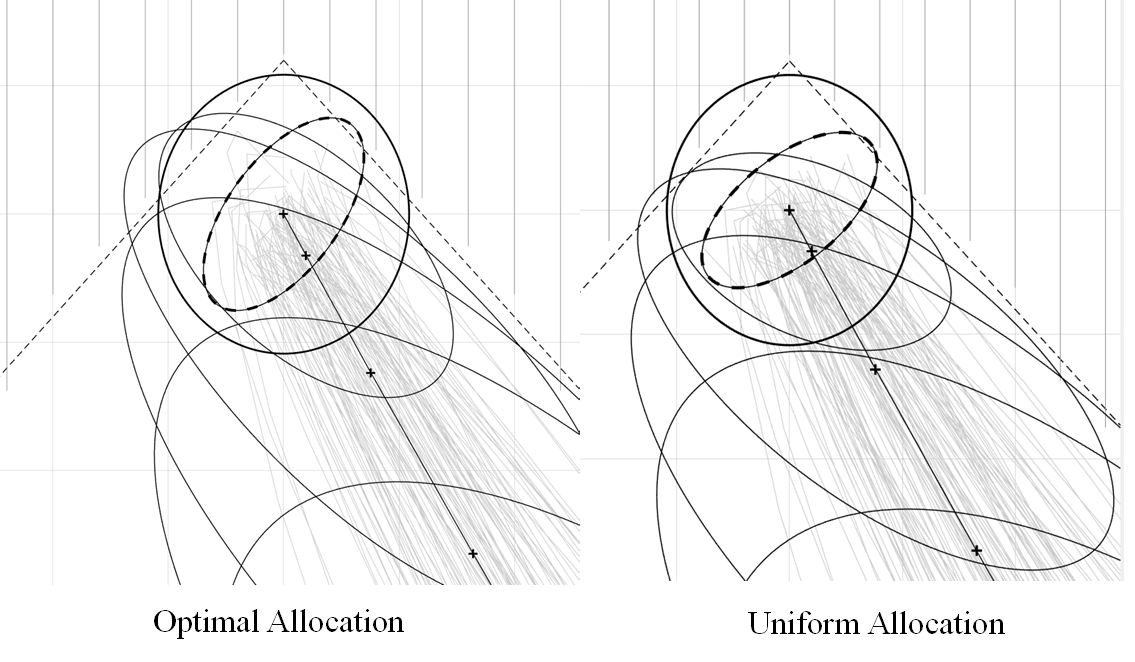}
	\caption{Comparison of terminal covariance steering using a uniform and the optimal risk allocation.}
	\label{fig:6}
\end{figure}

\begin{figure}[!htb]
	\centering
	\hspace*{-0.8cm}
	\includegraphics[scale=0.49]{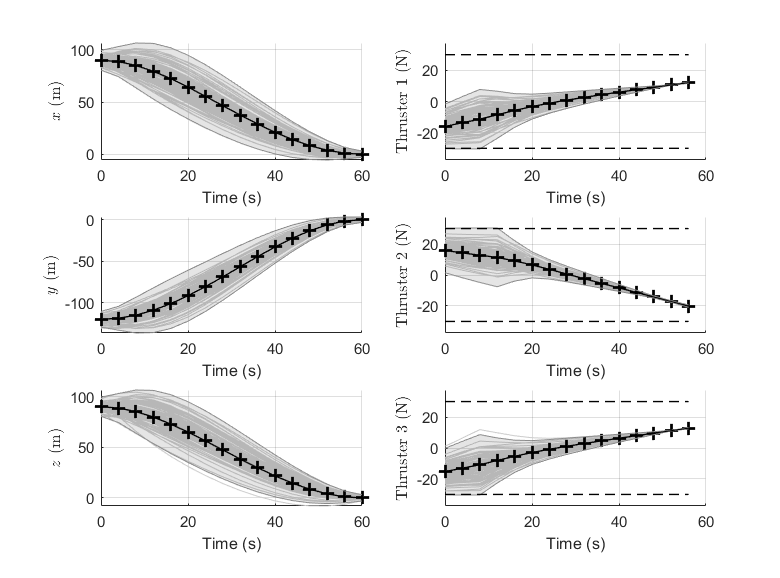}
	\caption{Trajectories of controlled system and their associated standard deviations using iterative risk allocation.}
	\label{fig:7}
\end{figure}

Figure~\ref{fig:7} shows the state trajectories and the optimal controls for the polyhedral chance constraints.
The control is almost linear but has a large variance in the first 10 time steps, where it may saturate due to the disturbances.
Figure~\ref{fig:9} shows the a priori allocation of risk, as well as the true risk $\bar{\delta}$ once the optimization is completed, where $\delta_r$ corresponds to the risk allocated for the right boundary and $\delta_u$ for the risk allocated for the top boundary. 
Notice that in Figure~\ref{fig:9}(a) the true risk exposure is \textit{much} lower than the allocated risk, which confirms the conclusion that the solutions for the uniform allocation case
can be overly conservative. 
Comparing to Figure~\ref{fig:9}(b), we see a close correspondence between the allocated risk and the true risk exposure over the whole horizon for the optimal risk allocation case.
It should be noted that although the true risk is still slightly less than the allocated risk, the error between the two is much smaller when compared to that of the uniform risk allocation strategy. 

\begin{figure}[!htb]
	\centering
	\begin{subfigure}{0.5\textwidth}
		\centering
		\hspace*{-0.8cm} 
		\includegraphics[scale=0.35]{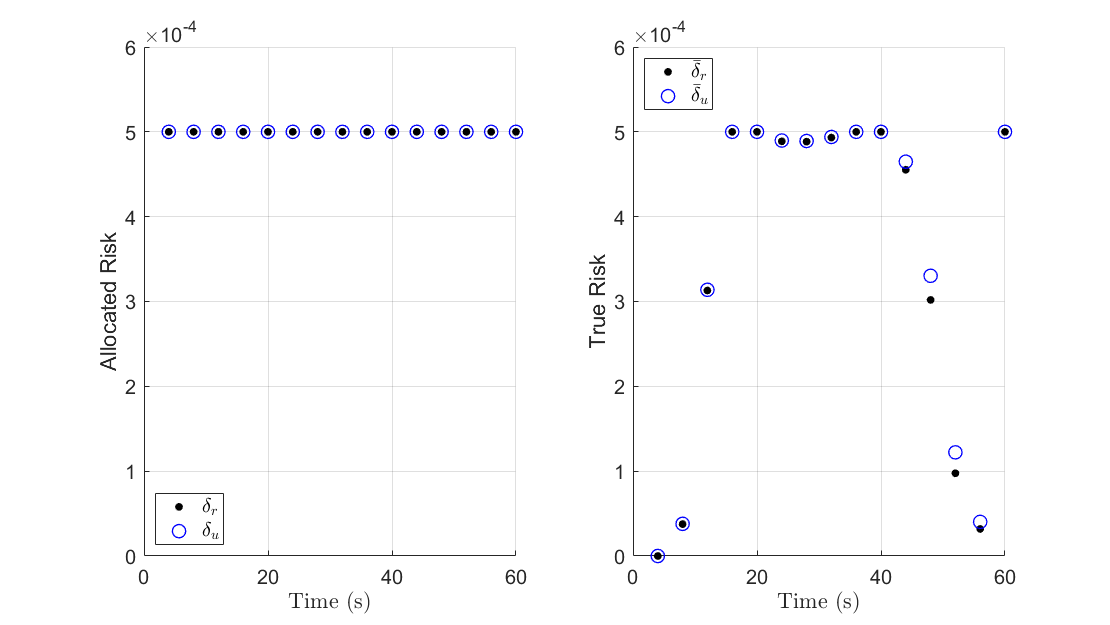}
		\captionof{figure}{Uniform allocation.}
		\label{fig:9a}
	\end{subfigure}\\
	\begin{subfigure}{0.5\textwidth}
		\centering
		\hspace*{-0.8cm} 
		\includegraphics[scale=0.35]{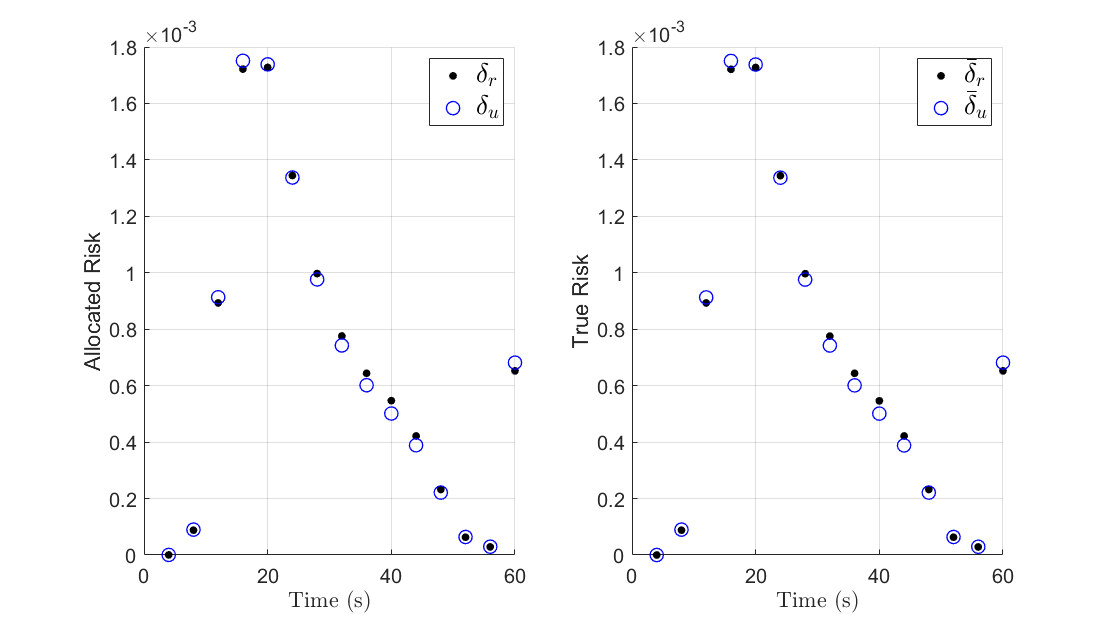}
		\captionof{figure}{Optimal allocation.}
		\label{fig:9b}
	\end{subfigure}
	\caption{Comparison of allocated risk and true risk using: (a) uniform risk allocation, (b) iterative risk allocation.}
	\label{fig:9}
\end{figure}

The iterative risk allocation algorithm is robust in the sense that the algorithm will assign risk proportionately to how close the solution trajectories are to the boundaries of the state space. 
Since solutions are close to the right and top boundaries of the allowable LOS region 
for most of the horizon, the optimal allocation weighs the respective risks more than those of the left and bottom boundaries.
Thus, IRA purposefully steers the trajectories away from the boundaries proportionally to the amount of risk that is allocated to violating those respective boundaries.
Table~\ref{table:1} shows the true joint probability of failure, defined as 
\begin{equation}~\label{eq:42}
	\bar{\Delta} \coloneqq 1 - \mathbb{P}\Bigg[\bigwedge_{k=1}^{N}\bigwedge_{j=1}^{M}\alpha_j^\intercal E_k X^* \leq \beta_j\Bigg].
\end{equation}
It is clear that the uniform risk allocation does not even come close to the desired design of $\Delta = 0.03$, while the IRA gives a true probability of failure very close to the desired one.
\begin{figure}[!htb]
	\centering
	\includegraphics[scale=0.37]{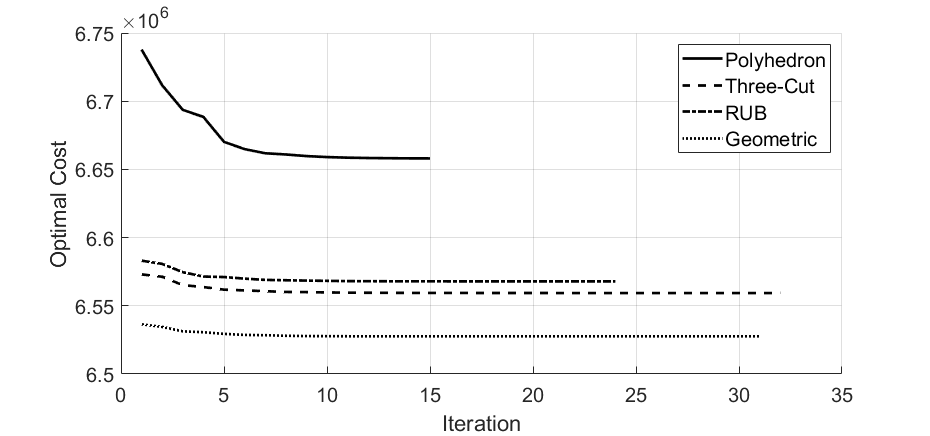}
	\caption{Optimal cost $J^*(\delta)$ after every IRA iteration.}
	\label{fig:10}
\end{figure}

\begin{table}[!htb]
	\centering
	\caption{Comparison of total true risk and terminal volume between uniform and optimal allocations.}
	{\tabulinesep=1.2mm
		\begin{tabu}{|c|c|c|c|c|}
			\hline
			- & Uniform & IRA Poly & IRA TC & IRA RUB \\
			\hline
			$\bar{\Delta}$ & 0.0123 & 0.02998 & 0.029990 & 0.029994 \\
			$V_N$ & 0.5546 & 0.6279 & 3.6818 & 3.7038 \\
			\hline 
			- & IRA GEO & CS-SMPC \cite{CSMPC} & SMPC \cite{MPC_Farina} &\\
			\hline
			$\bar{\Delta}$ & 0.029979 & 0.01128 & 0.00105 &\\
			$V_N$ & 4.0909 & - & - & \\
			\hline
	\end{tabu}}
	\label{table:1}
\end{table}

Finally, we looked at the optimal cost function over each IRA iteration, as in Figure \ref{fig:10}. 
The convergence criterion set in this example is $\epsilon = 10^{-5}$, or when all of the constraints are active or inactive, which can be proved in~\cite{b1} to 
be a sufficient condition for optimality for Problem~3. 
We see that indeed (\ref{eq:29}) holds, and the optimization resulted even in a slight decrease of the objective function.
Thus, the iterative risk allocation algorithm optimizes the risk allocation at each time step without increasing the cost.

\subsection{IRA-CS with Cone Chance Constraints}

%For the convex cone chance constraint case, we also implemented the method outlined in Section~IV, 
%namely the $4N$ constraints in (\ref{eq:39}). 
For this example, the following representation of a cone was used
\begin{equation}~\label{eq:43}
	\mathcal{X}^c = \{(x,y,z): \|A\bar{R}_{x}(\psi)x\|_2 \leq c^\intercal \bar{R}_{x}(\psi) x + d\},
\end{equation}
where $\bar{R}_{x}(\psi) := \textrm{blkdiag}(R_{x}(\psi), R_{x}(\psi))$ and $R_{x}(\psi) \in \mathrm{SO(3)}$ is the standard 3D rotation matrix. 
The angle of rotation is $\psi = \tan^{-1}(\bar{z}_0/\bar{y}_0)$, which corresponds to a body-mounted sensor on the Chief that is angled to the relative position vector of the Deputy.
Additionally, $A = \textrm{diag}(1,0,1,0,0,0),\ c = [0,\lambda,0]^\intercal$, and $\lambda = \tan(\theta)$, where $\theta = 15^\circ$ is the cone half-angle.
Lastly, $d = 10$, which corresponds to an offset of 10 meters from the origin.
\begin{comment}
This requirement translates to the individual chance constraints
\begin{equation}~\label{eq:44}
\mathbb{P} \big( x_k^2 + z_k^2 \leq (\lambda y_k)^2 \big) \geq 1 - \delta_k, \quad k = 1,\ldots,N.
\end{equation}
As discussed in Section~\ref{sec:quad}, the probabilistic constraint in (\ref{eq:44}) is relaxed to $\kappa = \lambda\bar{y}_k$, so that at each time step, the state is forced to stay inside a disk with radius $\lambda\bar{y}_k$. 
%This relaxation is needed because $y_k$ is a random variable, and only a deterministic $\kappa$ can be used in (\ref{eq:36}). 
\end{comment}
For the geometric approximation, we set $A_c = \textrm{diag}(1,0,1), h_1 = [1,0,0]^\intercal, h_2 = [0,0,1]^\intercal$, and $c_c = [0,\lambda,0]^\intercal$.
The constant choice of parameters $\beta_i = 1/n$ for all $i=1,\ldots, n$ was used across all constraints in the three-cut approximation and $\epsilon_1=\epsilon_2 = 1-\beta_i \delta_k/2$ was used for the reverse union bound approximation.
Lastly, the initial state was changed to $\mu_0 = [10,120,90,0,0,0]^\intercal$ for this simulation.

It should be noted that a rotated cone about the $i$th axis can be put in the form of a standard cone as in (\ref{eq:34}) via the transformation
\begin{equation}
	\tilde{A} := A\bar{R}_{i}(\psi), \quad \tilde{c} := \bar{R}^\intercal_{i}(\psi) c.
\end{equation}
Thus, it is possible to make successive rotations of a cone by adjusting the cone parameters $A$ and $c$.
In the context of the given parametrization, the cone becomes the set
\begin{equation}
	\bigg\|\begin{bmatrix}
		p_x \\
		p_y\sin\psi + p_z\cos\psi
	\end{bmatrix}
	\bigg\|_2 \leq \gamma(p_y\cos\psi - p_z\sin\psi) + d.
\end{equation}
\vspace*{-0.5cm}
\begin{figure}[!htb]
	\centering
	\begin{subfigure}[t]{.24\textwidth}
		\centering
		\hspace*{-0.5cm} 
		\includegraphics[scale=0.3]{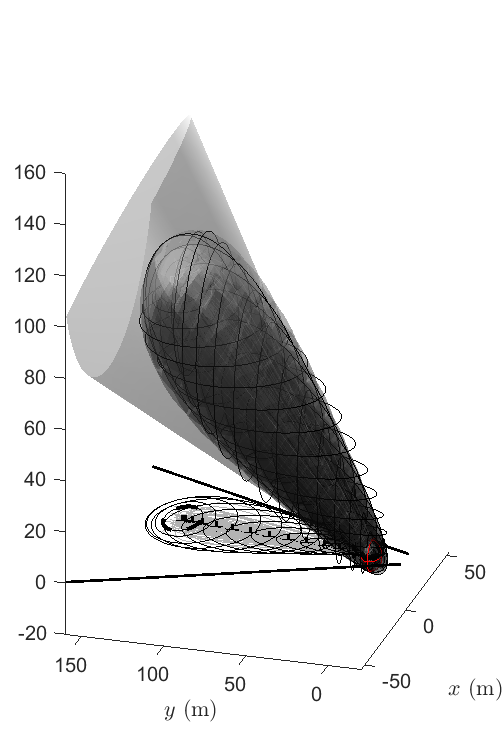}
		\captionof{figure}{Reverse union bound two-sided approximation.}
		\label{fig:11b}
	\end{subfigure}%
	\begin{subfigure}[t]{.24\textwidth}
		\centering
		\hspace*{-0.4cm}
		\vspace*{0.4cm}
		\includegraphics[scale=0.31]{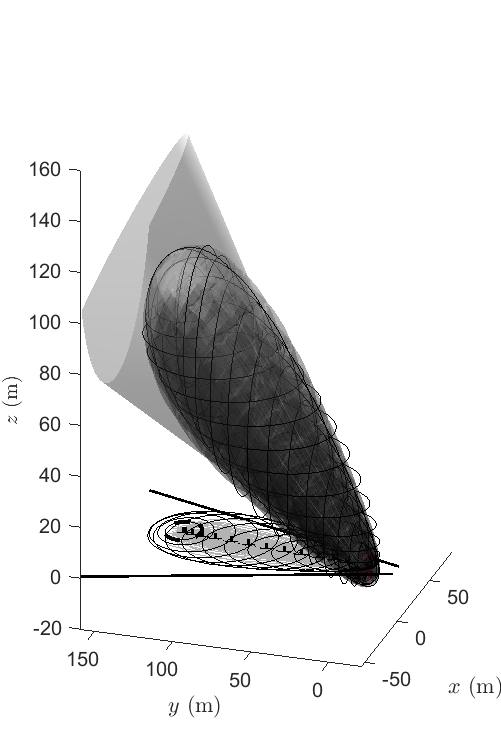}
		\captionof{figure}{Geometric approximation.}
		\label{fig:11a}
	\end{subfigure}
	\caption{Optimal trajectories using convex relaxation of conic chance constraints.}
	\label{fig:11}
\end{figure}
%
%\begin{figure}[!htb]
%	\centering
%	\begin{subfigure}{.24\textwidth}
%		\centering
%		\hspace*{-0.5cm} 
%		\includegraphics[scale=0.2]{RPO2DSolCC_TS_Revised}
%		\captionof{figure}{Reverse union bound two-sided approximation.}
%		\label{fig:12a}
%	\end{subfigure}%
%	\begin{subfigure}{.24\textwidth}
%		\centering
%		\hspace*{-0.27cm}
%		\vspace*{0.15cm}
%		\includegraphics[scale=0.19]{RPO2DSolCC_GEO_Revised}
%		\captionof{figure}{Geometric approximation.}
%		\label{fig:12b}
%	\end{subfigure}
%	\caption{Optimal planar trajectories using convex relaxation of conic chance constraints.}
%	\label{fig:12}
%\end{figure}
%
Figure~\ref{fig:10} compares the results of the three approximations, with the geometric one being the best and the three-cut approximation being less conservative than the RUB approximation, as expected.
Figure~\ref{fig:11} shows the optimal trajectories in the three-dimensional space and in the projection on the $x$-$y$ plane, respectively. 
It should be noted that for both the three-cut and RUB  approximations of the cone constraint, and since we approximated the quadratic constraints as $\mathcal{O}(n)$ SOC constraints at each time step, the IRA algorithm needs to be adjusted as follows.
In Line 5 of Algorithm 1, a constraint is active at time step $k$ if \textit{any} of the constraints in (\ref{eq:two_sided_constraints_1}) are active. 
Similarly, when tightening the constraints in Line~10, the \textit{maximum} true risk $\bar{\delta}_k := \max_j \bar{\delta}_k^j$ is used. 
This is not needed for the geometric approximation because it approximates each cone chance constraint as a \textit{single} convex constraint for each $k$, so the standard IRA algorithm is applicable.

From Table~\ref{table:1}, the geometric approximation actually has the highest terminal covariance of all four IRA methods discussed. 
Other than that, the two optimal inputs and trajectories are very similar from Figure 8, and both successfully steer the distribution of states to the terminal distribution, while satisfying the cone constraints.

\subsection{Comparison with MPC-based Controllers}

In addition to comparing the presented methods with each other, it is also worthwhile to compare these methods to MPC-based methods in the context of the spacecraft rendezvous problem. 
Although there exists a significant literature on the use of MPC-based approaches for satellite rendezvous and proximity operations in space~\cite{MCLDA17,WBEK15,CPK12,LRG14,PZVZKR16}, most of these results assume deterministic dynamics and do not handle directly chance constraints.
To this end, we applied two recent \textit{stochastic} MPC (SMPC) formulations, outlined in~\cite{CSMPC} and \cite{MPC_Farina}, and compared them to the present formulation.
Note that, with the exception of \cite{CSMPC} and \cite{MPC_Farina}, most SMPC methods \cite{robustMPC1,robustMPC2,MCLDA17} assume \textit{bounded} disturbances and/or chance constraints on the input, while in the present case, we assume more general, unbounded disturbances with hard constraints on the input, making the problem much harder to solve as a result.

\begin{figure}[!htb]
	\centering
	\begin{subfigure}{.24\textwidth}
		\centering
		\hspace*{-0.45cm} 
		\includegraphics[scale=0.35]{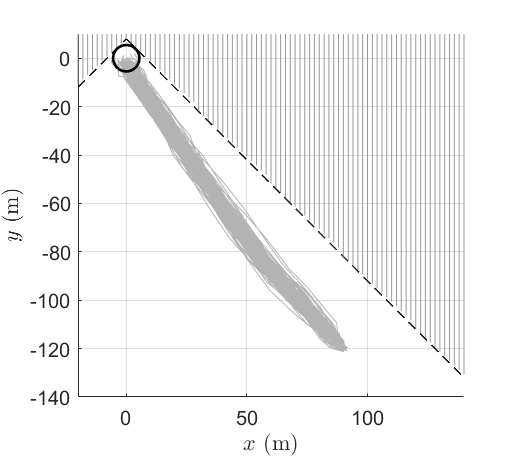}
		\captionof{figure}{Covariance steering SMPC \cite{CSMPC}.}
		\label{fig:13a}
	\end{subfigure}%
	\begin{subfigure}{.24\textwidth}
		\centering
		\vspace*{0.1cm}
		\hspace*{-0.2cm}
		\includegraphics[scale=0.35]{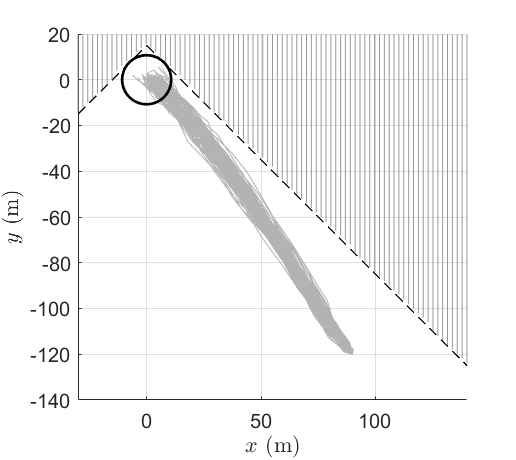}
		\captionof{figure}{Tube SMPC \cite{MPC_Farina}.}
		\label{fig:13b}
	\end{subfigure}
	\caption{Stochastic MPC methods in the rendezvous problem.}
	\label{fig:13}
\end{figure}

Figure~\ref{fig:13} shows the result from the MPC-based methods in \cite{CSMPC} and \cite{MPC_Farina}. 
In \cite{CSMPC}, the authors iteratively solve a covariance steering problem over a receding horizon.
There are two terminal constraints at every time step; the first is on the mean state to be inside a maximally positive invariant (MPI) set defined by $(A+B\tilde{K})\mu\in\mathcal{X}_f^{\mu}$, where $\tilde{K}$ is the associated gain that satisfies $\Sigma = (A+B\tilde{K})\Sigma(A+B\tilde{K})^\intercal + DD^\intercal$, and $\Sigma$ is an \textit{assignable} state covariance \cite{CovarianceAssignmentTheory}. 
The second, and similar to this work, is a constraint on the terminal covariance to be less than or equal to $\Sigma_f$, where $\Sigma_f$ is the aforementioned assignable covariance.
Similarly, in \cite{MPC_Farina} there are again two terminal constraints on the mean and the covariance, but in this case, the MPI set is defined as $(A+BK_{\textrm{LQR}})\mu\in\mathcal{X}_{f}^{\mu}$, where $K_{\textrm{LQR}}$ is the gain computed as the solution of an LQ problem for the nominal system, and $\Sigma_f$ is chosen to solve the Lyapunov equation $\Sigma_f = (A+BK_{\textrm{LQR}})\Sigma_f(A+BK_{\textrm{LQR}})^\intercal + DD^\intercal$. 

In \cite{MPC_Farina}, the algorithm was \textit{not} able to solve the rendezvous problem with the given parameters for a control bound of $u_{\textrm{max}} = 30$ N, as this was not enough control authority to reach the MPI set at the end of the horizon ($N = 10, \Delta t = 2$).
The MPC problems became feasible only when the bounds were relaxed to more than five times that of the current work. 
Additionally, in both SMPC formulations, the designer is not allowed to freely choose the terminal covariance, since in order for the problem to be feasible, the terminal covariance must satisfy certain constraints. 
Lastly, although both MPC controllers successfully steer the system to the origin, there are no guarantees of recursive feasibility for the closed loop system and the solutions are more sub-optimal compared to that of the current work, as there is no optimization over the risks since they are uniformly allocated.
As a result, the \textit{true} risk in these MPC methods is 
much lower than the design goal, as shown in Table I.
Further, there is no current work to date that can incorporate MPC-based methods for conical state spaces as was proposed here.
As such, one potential future application of the proposed CS-IRA algorithm is in the context of SMPC, which would lead to less conservative and more optimal trajectory designs, without having to uniformly allocate risk or use a polyhedral approximation of the LOS cone.

\section{Conclusion}

In this paper, we have incorporated an iterative risk allocation (IRA) strategy to optimize the probability of violating the state constraints at every time step within the covariance steering problem of a linear stochastic system subject to chance constraints.
For the covariance steering problem, we showed that employing IRA not only leads to less conservative solutions that are more practical, but also tends to maximize the final covariance. 
Additionally, the use of IRA in the context of CS with chance constraints results in optimal solutions that have a true risk much closer to the intended design requirements, compared to the use of a uniform risk allocation. 
We also implemented quadratic chance constraints in the form of convex cones, which are more accurate and natural for many engineering applications. 
Using two different relaxation methods these quadratic chance constraints can be made convex and the IRA algorithm can be used to optimize the risk.
%The solutions from the ensuing SDP were similar to the polytopic case, which shows that quadratic constraints 
%could be used to describe the feasible state space. 
Lastly, we also propose a geometric approximation of the cone chance constraints, which is valid when the constraint space is three dimensional, as is often the case when constraining the position of a vehicle, and which is less conservative than either of the two two-sided approximations. 
All proposed relaxations result in convex programs, where the two-stage IRA algorithm is applicable.

% In this paper, we assumed an LTV system, but this work can be extended to nonlinear systems using iterative covariance steering (iCS)\cite{b2}. Further work can look at the applications of IRA to the problems of aircraft landing and powered-descent guidance of a rocket, where the state constraints represents a three dimensional region in space, as opposed to the rendezvous problem, which could be decoupled into planar motion. Lastly, we note that there could be potential extensions of CS under chance constraints with risk allocation to include free final time problems, as is often the case in a practical setting where you want to get from an initial distribution to a terminal distribution, but not necessarily for an apriori fixed time horizon.

\section{Acknowledgment}

This work has been supported by NASA University Leadership Initiative award 80NSSC20M0163 and ONR award N00014-18-1-2828.
The first author also acknowledges support from the Georgia Tech School of Aerospace Engineering Graduate Fellows Program.
The article solely reflects the opinions and conclusions of its authors and not any NASA entity.

\bibliography{Refs} 
\bibliographystyle{IEEEtran}
%\newpage
%\onecolumn

\appendix

\section*{A.~Cone Chance Constraint Relaxation}

\setcounter{equation}{0}
\renewcommand{\theequation}{A.\arabic{equation}}

\begin{lem}
	The quadratic chance constraint
	\begin{equation}  \label{app:eq1}
		\mathbb{P}(\|Ax + b\|_2 \leq c^\intercal \mu + d) \geq 1 - \delta, 
	\end{equation} 
	where $x \sim \mathcal{N}(\mu, \Sigma)$
	%where $\bar{x} = \mathbb{E}[x]$,
	is a relaxation of  the cone chance constraint
	\begin{equation}    \label{app:eq2}
		\!\mathbb{P}(\|Ax + b\|_2 \leq c^\intercal x + d) \geq 1 - \delta.
	\end{equation}  
\end{lem}

\begin{proof}
	Since $x \sim \mathcal{N}(\mu,\Sigma)$ it follows that $\xi := \|Ax + b\|_2$ follows a non-central $\chi^2$ distribution with probability density function $f_{\xi}(x)$~\cite{stats}.
	Let $\eta := c^\intercal x + d$ and notice that  $\eta \sim \mathcal{N}(c^\intercal\mu+d,c^\intercal\Sigma c)$. 
	The chance constraint (\ref{app:eq2}) then takes the form $\mathbb{P}(\xi \leq \eta) \geq 1 - \delta$. 
	The probability that one random variable is less than or equal another random variable is given by
	\begin{equation}  \label{app:eq3}
		\mathbb{P}(\xi \leq \eta) = \int_{-\infty}^{\infty}\int_{-\infty}^{y} f_{\xi,\eta}(x,y) \,\d x \,  \d y,
	\end{equation}
	where $f_{\xi,\eta}(x,y)$ is the \textit{joint} probability distribution function of the random variables $\xi$ and $\eta$.
	%The first integral on the right is over the domain of $\xi$ while the second is over the domain of $\eta$. 
	Next, 
	%since $\eta$ is a non-standard Gaussian, 
	write $\eta = c^\intercal \mu + d + z \sqrt{c^\intercal \Sigma c}$, where $z \sim \mathcal{N}(0,1)$ with probability density $f_z(y)$ and let $\bar{\eta} = c^\intercal \mu + d$.
	The inner integral in (\ref{app:eq3}) then becomes
	\begin{align*}
		\int_{-\infty}^{y} &f_{\xi,\eta}(x,y) \,\d x \, = \int_{-\infty}^{\bar{\eta} + y \sqrt{c^\intercal \Sigma c}}f_{\xi, z}(x,y) \, \d x \\
		&= \int_{-\infty}^{\bar{\eta}}f_{\xi,z}(x,y) \, \d x + \int_{\bar{\eta}}^{\bar{\eta} + y \sqrt{c^\intercal\Sigma c}} f_{\xi,z}(x,y) \, \d x.
	\end{align*}
	It follows that
	\begin{align} \label{app:eq30}
		&\mathbb{P}(\xi \leq \eta) = \int_{-\infty}^{\infty}\int_{-\infty}^{\bar{\eta}}f_{\xi,z}(x,y) \, \d x \, \d y \nonumber \\
		&\qquad\qquad\qquad+  \int_{-\infty}^{\infty}\int_{\bar{\eta}}^{\bar{\eta} + y \sqrt{c^\intercal\Sigma c}} f_{\xi,z}(x,y) \, \d x \, \d y \nonumber \\
		&\ \quad\qquad\qquad \ \ \geq \int_{-\infty}^{\infty}  \int_{-\infty}^{\bar{\eta}}f_{\xi,z}(x,y) \, \d x \, \d y.
	\end{align}
	Noticing that 
	\begin{equation*}
		\int_{-\infty}^{\infty}  f_{\xi,z}(x,y) \, \d y = f_\xi(x),
	\end{equation*}
	%By definition of a probability density function, it is always non-negative, which implies
	the last expression in (\ref{app:eq30}) implies that
	\begin{equation*}
		\mathbb{P}(\xi \leq \eta) \geq   \int_{-\infty}^{\bar{\eta}}   f_\xi(x) \, \d x  = \mathbb{P}(\xi \leq \bar{\eta}),
	\end{equation*}
	which achieves the desired result.
	In order words, 
	if the relaxed chance constraint $\mathbb{P}(\xi \leq \bar{\eta}) \geq 1 - \delta$ is satisfied, then the \textit{original} chance constraint 
	$\mathbb{P}(\xi \leq \eta) \geq 1 - \delta$
	is satisfied as well.
\end{proof}

\section*{B.~Reverse Union Bound}

\setcounter{equation}{0}
\renewcommand{\theequation}{B.\arabic{equation}}

\begin{lem}
	Let the events $A_1,\ldots,A_n$, such that $\mathbb{P}(A_i) \geq \delta_i$, for some $\delta_i \ge 0$, for all $i = 1,\ldots,n$.
	Then,
	\begin{equation}
		\mathbb{P}\left(\ \bigcap_{i=1}^{n} A_i \right) \geq \sum_i \delta_i - (n-1). \label{eq:Fact1}
	\end{equation}
\end{lem}

\begin{proof}
	From the law of total probability, we have that
	\begin{equation}  \label{eq:C1}
		\mathbb{P}\left(\ \bigcap_{i=1}^{n} A_i \right) = 1 - \mathbb{P}\left(\bigcup_{i=1}^{n} {A}^c_i\right), 
	\end{equation}
	where ${A}^c_i$ denotes the complement of the event $A_i$. 
	Note that 	$\mathbb{P} (A^c_i) \leq 1 - \delta_i$ for all $i = 1,\ldots,n$.
	% 	From De Morgan's law,
	% 	\begin{equation}
	% 		\mathbb{P}\left( {\bigcap_{i=1}^{n} A_i} \right)^c = \mathbb{P}\left(\bigcup_{i=1}^{n} {A}^c_i\right). \label{eq:C2}
	% 	\end{equation}
	From Boole's inequality, it follows that
	\begin{equation}
		\mathbb{P}\left(\bigcup_{i=1}^{n} {A}^c_i\right) \leq \sum_{i=1}^{n} \mathbb{P}({A}^c_i). \label{eq:C3}
	\end{equation}
	Combining (\ref{eq:C1}) - (\ref{eq:C3}) yields 
	\begin{align*}
		\mathbb{P}\left( {\bigcap_{i=1}^{n} A_i} \right) &\geq 1 - \sum_{i=1}^{n} \mathbb{P}({A}^c_i) \\
		&\geq 1 - \sum_{i=1}^{n} (1 - \delta_i) \\
		&= \sum_{i=1}^{n} \delta_i - (n - 1),
	\end{align*}
	which achieves the desired result.
\end{proof}

%\begin{comment}
\section*{C.~Input Hard Constrained Covariance Controller}

\setcounter{equation}{0}
\renewcommand{\theequation}{C.\arabic{equation}}

We assume that the hard input constraints on the control are affine, i.e., they are of the form 
\begin{equation}
	\alpha_{u,s}^\intercal F_k U \leq \beta_{u,s}, \quad s = 1,\ldots,N_c.
\end{equation}
\begin{theorem}[\bf \!\!\cite{IH}]
	The control law
	\begin{equation}
		u_k = v_k + K_k z_k,
	\end{equation}
	where $z_k$ is governed by the dynamics 
	\begin{align}
		z_{k+1} &= Az_k + \phi(w_k), \\
		z_0 &= \phi(y_0), \quad y_0 = x_0 - \mu_0,
	\end{align}
	where $\phi : \mathbb{R}^n \rightarrow \mathbb{R}^n$ is an element-wise symmetric saturation function with pre-specified saturation value of the $i$th entry of the input $y_i^{\mathrm{max}} > 0$ as
	\begin{equation}
		\phi_i(y) = \mathrm{max}(-y_i^{\mathrm{max}},\mathrm{min}(y_i,y_i^\mathrm{max})),
	\end{equation}
	converts Problem~1 
	to the following convex programming problem that constrains the control to a maximum saturation value
	\begin{align}
		&\min_{V,K,\Omega} J(V,K,\Omega) = \textrm{tr}\bigg(\bar{Q} 
		\begin{bmatrix}
			I & \mathcal{B} K
		\end{bmatrix} 
		\Sigma_{XX} 
		\begin{bmatrix}
			I \\
			K^\intercal \mathcal{B}^\intercal
		\end{bmatrix}\bigg) \nonumber\\
		&+ \mathrm{tr}(\bar{R}K\Sigma_{UU}K^\intercal) + (\mathcal{A}\mu_0 + \mathcal{B} V)^\intercal\bar{Q}(\mathcal{A}\mu_0 + \mathcal{B} V) + V^\intercal \bar{R}V \nonumber\\
		&\textrm{subject to} \nonumber\\
		& \mathbb{P}(E_k X \notin \mathcal{X}) \leq \delta_k, ~~ k = 1,\ldots, N \label{eq:CCgeneral}\\
		& \sum_{k = 1}^{N} \delta_k \leq \Delta, \\
		& H F_k V + \Omega^\intercal \sigma \leq h, \\
		& H F_k K 
		\begin{bmatrix}
			\mathcal{A} & \mathcal{D}
		\end{bmatrix}
		= \Omega^\intercal S, \\
		& \Omega \geq 0, \\
		& \mu_f = E_N(\mathcal{A}\mu_0 + \mathcal{B}V), \\ 
		& \Sigma_f \geq E_N
		\begin{bmatrix}
			I & \mathcal{B}K
		\end{bmatrix}
		\Sigma_{XX}
		\begin{bmatrix}
			I\\
			K^\intercal\mathcal{B}^\intercal
		\end{bmatrix}
		E_N^\intercal,
	\end{align}
	where $\Omega \in \mathbb{R}^{2(N+1)n\times N_c}$ is a decision (slack) variable,
	\begin{align}
		\Sigma_{XX} &= 
		\begin{bmatrix}
			\mathcal{A}& \\
			& \mathcal{A}
		\end{bmatrix}
		\begin{bmatrix}
			\Sigma_0 & \mathbb{E}[y_0\phi(y_0)^\intercal]\\
			\mathbb{E}[\phi(y_0)y_0^\intercal] & \mathbb{E}[\phi(y_0)\phi(y_0)^\intercal]
		\end{bmatrix}
		\begin{bmatrix}
			\mathcal{A}^\intercal & \\
			& \mathcal{A}^\intercal
		\end{bmatrix} \nonumber\\
		&+ 
		\begin{bmatrix}
			\mathcal{D}& \\
			& \mathcal{D}
		\end{bmatrix}
		\begin{bmatrix}
			I & \mathbb{E}[W\phi(W)^\intercal]\\
			\mathbb{E}[\phi(W)W^\intercal] & \mathbb{E}[\phi(W)\phi(W)^\intercal]
		\end{bmatrix}
		\begin{bmatrix}
			\mathcal{D}^\intercal & \\
			& \mathcal{D}^\intercal
		\end{bmatrix},
	\end{align}
	\begin{equation}
		\Sigma_{UU} = \mathcal{A}\mathbb{E}[\phi(y_0)\phi(y_0)^\intercal]\mathcal{A}^\intercal + \mathcal{D}\mathbb{E}[\phi(W)\phi(W)^\intercal]\mathcal{D}^\intercal.
	\end{equation}
	Further,
	\begin{align}
		H &= [\alpha_{u,1}, \ldots, \alpha_{u,N_c}]^\intercal \in \mathbb{R}^{N_c \times m}, \\
		h &= [\beta_{u,1}, \ldots, \beta_{u,N_c}]^\intercal \in \mathbb{R}^{N_c}. 
	\end{align}
	In addition, $S \in \mathbb{R}^{2(N+1)n \times (N+1)n}$ and $\sigma\in\mathbb{R}^{2(N+1)n}$ are constant, given by
	\begin{align}
		S_{2i-1} &= e_{2i-1}^\intercal, \ S_{2i} = -e_{2i}^\intercal, \\ 
		\sigma_{2i-1} &= y_i^{\mathrm{max}}, \ \sigma_{2i} = y_i^{\mathrm{max}},
	\end{align}
	where $S_i$ denotes the $i$th row of $S$, and $e_i\in\mathbb{R}^{2(N+1)n}$ is a unit vector with $i$th element 1.
\end{theorem}

%\begin{proof}
%	It can be shown that 
%	\begin{equation}
%		\mathbb{E}[\tilde{X}\tilde{X}^\intercal] = 
%		\begin{bmatrix}
%			I & \mathcal{B}K
%		\end{bmatrix}
%		\Sigma_{XX}
%		\begin{bmatrix}
%			I \\
%			K^\intercal \mathcal{B}^\intercal
%		\end{bmatrix},
%	\end{equation}
%	from which it follows that the standard deviation of the state vector $x_k = E_k X$ is 
%	\begin{equation}
%		\sqrt{E_k \mathbb{E}[\tilde{X}\tilde{X}^\intercal] E_k^\intercal} = \|\Sigma_{XX}^{1/2}
%		\begin{bmatrix}
%			I & \mathcal{B}K
%		\end{bmatrix}^\intercal E_k^\intercal \|.
%	\end{equation}
%	The rest of the proof then follows from the results in \cite{IH}.
%\end{proof}

\medskip

Note that the saturation of the input will result, in general, in a non-Gaussian distribution of the state.
As a result, the chance constraint inequalities (\ref{eq:24}) 
must be replaced by another set of inequalities, for example, of the Chebyshev-Cantelli type~\cite{MaOl:1960}.
More details can be found in~\cite{IH}.
For the spacecraft rendezvous problem in Section~\ref{sec:NumEx}, it turned out, however, that the original chance constraint formulation worked well, which means that the  Chebyshev-Cantelli inequality formulation of the chance constraints may be overly conservative for this problem.

\end{document}